\documentclass[11pt]{amsart}
\usepackage{amsmath, amssymb}
\usepackage{amsfonts}
\usepackage[arrow,matrix,curve,cmtip,ps]{xy}

\usepackage{amsthm}

\allowdisplaybreaks

\newtheorem{theorem}{Theorem}[section]
\newtheorem{lemma}[theorem]{Lemma}
\newtheorem{proposition}[theorem]{Proposition}
\newtheorem{corollary}[theorem]{Corollary}

\newtheorem*{theorem*}{Theorem}
\theoremstyle{remark}
\newtheorem{remark}[theorem]{Remark}
\newtheorem{definition}[theorem]{Definition}
\newtheorem{example}[theorem]{Example}

\numberwithin{equation}{section}

\newcommand{\N}{\mathbb{N}}
\newcommand{\R}{\mathbb{R}}
\newcommand{\C}{\mathbb{C}}

\newcommand{\dec}{\textnormal{dec}}

\newcommand{\im}{\operatorname{im }}

\newcommand{\spn}{\operatorname{span}}

\begin{document}
\title{vector spaces with an order unit}

\author{Vern I. Paulsen}

\address{Department of Mathematics \\ University of Houston \\ Houston, TX 77204-3008 \\USA}
\email{vern@math.uh.edu}

\author{Mark Tomforde}

\address{Department of Mathematics \\ University of Houston \\ Houston, TX 77204-3008 \\USA}
\email{tomforde@math.uh.edu}

\thanks{The first author was supported by NSF Grant DMS-0600191}

\date{\today}



\begin{abstract}
We develop a theory of ordered $*$-vector spaces with an order unit.  We prove fundamental results concerning positive linear functionals and states, and we show that the order (semi)norm on the space of self-adjoint elements admits multiple extensions to an order (semi)norm on the entire space.  We single out three of these (semi)norms for further study and discuss their significance for operator algebras and operator systems.  In addition, we introduce a functorial method for taking an ordered space with an order unit and forming an Archimedean ordered space.  We then use this process to describe an appropriate notion of quotients in the category of Archimedean ordered spaces.
\end{abstract}

\maketitle

\tableofcontents

\section{Introduction}

Kadison proved in \cite{Kad} that every ordered real vector space with an Archimedean order unit can be represented as a vector subspace of the space of continuous real-valued functions on a compact Hausdorff space via an order-preserving map that carries the order unit to the constant function 1. Such subspaces are generally referred to as \emph{function systems}, and Kadison's result therefore provides a representation theorem for real function systems. 

Kadison's representation theorem motivated Choi and Effros \cite{CE2} to obtain an analogous representation theorem for self-adjoint subspaces of unital C*-algebras that contain the unit. Such subspaces, together with a natural \emph{matrix order}, are called \emph{operator systems}.   Operator systems are complex vector spaces with a $*$-operation, and we refer to these as \emph{$*$-vector spaces}.  

In any $*$-vector space $V$, the self-adjoint elements $V_h$ form a real subspace and one may define an ordering on $V$ by specifying a cone of positive elements $V^+ \subseteq V_h$.  When $(V,V^+)$ is an ordered $*$-vector space, then $(V_h, V^+)$ is an ordered real vector space.  In addition, an element $e \in V$ is an order unit (respectively, an Archimedean order unit) for $(V,V^+)$ if it is an order unit (respectively, an Archimedean order unit) for $(V_h, V^+)$.  The axioms for an operator system require, among other things, that it is an ordered $*$-vector space with Archimedean order unit $e$ and that for each $n \in \N$ the $*$-vector space $M_n(V)$ is ordered with Archimedean order unit $e_n := \left( \begin{smallmatrix} e & & \\ & \ddots & \\  & & e \end{smallmatrix} \right)$.  

Choi and Effros's characterization of operator systems is a fundamental result in the subject of operator spaces and completely bounded maps, and operator systems have also been useful tools in a number of other areas.  Since one can embed an operator space as a corner of an operator system (see, for example, \cite[Theorem~13.4]{Pau}) operator systems often allow one to take a problem involving norms and rephrase it as a more tractable problem involving positivity.  

Despite their utility in other areas, however, there has been very little work done to create a theory for operator systems themselves.  One possible reason for this is that, although there is a wealth of information for ordered real vector spaces \cite{Alf, KG, KN, Peres, Sch}, almost no attention has been paid to ordered complex vector spaces and ordered $*$-vector spaces.  Since operator systems are ordered $*$-vector spaces at each matrix level, a detailed theory of ordered $*$-vector spaces seems to be a fundamental prerequisite for the analysis of operator systems.

In this paper we study ordered $*$-vector spaces.  We do this with
an eye toward operator systems, and lay the groundwork to develop a
categorical theory of operator systems.  With this in mind, there will
be two important themes in our work:  First, we will establish results
for ordered $*$-vector spaces which parallel those in the theory of
ordered real vector spaces.  In some cases, such as with the positive
linear functionals and states, the results in the $*$-vector space
setting will be similar to those in the real setting.  But in other
cases, such as with the order (semi)norm, the $*$-vector space setting
will exhibit new phenomena not seen in the real case.   The second
theme throughout our work will be to focus on ordered spaces with an
order unit that is not necessarily Archimedean.  This is important
because, as we shall see in our analysis, quotients of ordered vector
spaces with Archimedean order units will have order units that need
not be Archimedean.  This is also important for operator systems.  For
example, Fannes, Naechtergaele, and Werner \cite{FNW} have recently associated to each finitely correlated state on a quantum spin chain a matrix-ordered space that satisfies all the axioms of an operator system except that the order units may not be Archimedean.  It is therefore desirable to have a functorial process for forming an Archimedean ordered space from an ordered space with an order unit.

This paper is organized as follows.  In \S\ref{Real-Vec-Spaces} we develop results for ordered real vector spaces with an order unit.  In particular, in \S\ref{subsec-positive-R-functionals} and \S\ref{seminorm-from-order-unit} we give self-contained proofs of some basic results and prove versions of some well-known theorems for Archimedean ordered spaces that will apply to ordered spaces having only an order unit.  For example, the Hahn-Banach Theorem for Archimedean ordered vector spaces is well known (see \cite[Corollary~2.1]{Kad}, for example) but we prove a version in Theorem~\ref{Hahn-Banach-Thm} that holds for ordered spaces with an order unit that is not necessarily Archimedean.  We also establish basic facts for the order norm of an Archimedean space and show that in the non-Archimedean case it is only a seminorm.  In \S\ref{real-Arch-subsec} we introduce a functorial process, called the \emph{Archimedeanization}, for forming an Archimedean ordered space from an ordered space with an order unit.  In \S\ref{Real-quotients-subsec} we discuss quotients in the category of ordered real vector spaces and the category of Archimedean ordered real vector spaces.

In \S\ref{Complex-Vec-Spaces} we develop results for ordered $*$-vector spaces that parallel those for ordered real vector spaces.  In \S\ref{C-functionals-subsec} we prove results for positive functionals and states.  In \S\ref{Arch-Complex-Vec-Spaces} we introduce a version of the Archimedeanization for ordered $*$-vector spaces, and in \S\ref{Complex-quotients-subsec} we consider quotients of ordered $*$-vector spaces.

In \S\ref{metric-from-order-unit} we consider seminorms on ordered $*$-vector spaces that are compatible with the order.  We examine the problem of extending the order (semi)norm on the self-adjoint elements to the entire space, and show that there are multiple ways to do this.  In particular, we find that there is a family of seminorms on a $*$-vector space extending the order seminorm, and we show that the seminorms in this family are all mutually boundedly equivalent, and hence, all induce the same topology. We single out three of these seminorms for more detailed study: the minimal order seminorm in \S\ref{minimal-seminorm-subsec}, the maximal order seminorm in \S\ref{maximal-seminorm-subsec}, and the decomposition order seminorm in \S\ref{decomp-seminorm-subsec}.

In \S\ref{Examples-sec} we examine examples of ordered vector spaces arising in the study of operator algebras, and we analyze the minimal, maximal, and decomposition norms.  In \S\ref{Funct-Sys-subsec} we prove a complex version of Kadison's Theorem, which shows that any Archimedean ordered $*$-vector space $V$ is unitally order isomorphic to a self-adjoint subspace of $C(X)$ containing the constant function 1 for some compact Hausdorff space $X$.  Furthermore, this embedding is an isometry with respect to the minimal order norm on $V$ and the sup norm on $C(X)$.  In \S\ref{unital-op-alg-subsec} we take a unital $C^*$-algebra $A$ and consider the ordered $*$-vector space $(A,A^+)$ with Archimedean order unit $I$.  We describe the minimal, maximal, and decomposition order norms on $A$ and relate them to the operator norm.  In \S\ref{min-max-equal-subsec} we show that for an Archimedean ordered $*$-vector space $V$ the maximal and minimal order seminorms are equal if and only if $V \cong \C$.  Finally, in \S\ref{convex-comb-subsec} we examine convex combinations of the maximal and minimal order norms to show that there are a continuum of order norms on any Archimedean ordered $*$-vector space other than $\C$ and to show that the decomposition order norm is not always a convex combination of the maximal and minimal order norms.

\section{Ordered real vector spaces} \label{Real-Vec-Spaces}

In this section we shall discuss ordered real vector spaces with an order unit and analyze their structure.  For such spaces we analyze the positive linear functionals, prove a version of the Hahn-Banach theorem that allows us to extend certain positive linear functionals, and describe the order (semi)norm on ordered vector spaces with an order unit.  In addition, we introduce a categorical method for taking a space with an order unit and constructing a space with an Archimedean order unit, and use this process to describe quotients of ordered real vector spaces.

\begin{definition} \label{real-ord-v-s-def}
If $V$ is a real vector space, a \emph{cone} in $V$ is a nonempty subset $C \subseteq V$ with the following two properties:
\begin{itemize}
\item[(a)] $a v \in C$ whenever $a \in [0,\infty)$ and $v \in C$
\item[(b)] $v + w \in C$ whenever $v, w \in C$.
\end{itemize}
An \emph{ordered vector space} $(V, V^+)$ is a pair consisting of a real vector space $V$ and a cone $V^+ \subseteq V$ satisfying
\begin{itemize}
\item[(c)] $V^+ \cap -V^+ = \{ 0 \}$.
\end{itemize}
In any ordered vector space we may define a partial ordering (i.e., a reflexive, antisymmetric, and transitive relation) $\geq$ on $V$ by defining $v \geq w$ if and only if $v - w \in V^+$.  Note that this partial ordering is translation invariant (i.e., $v \geq w$ implies $v+x \geq w + x$) and invariant under multiplication by non-negative reals (i.e., $v \geq w$ and $a \in [0, \infty)$ implies $av \geq aw$).  Also note that $v \in V^+$ if and only if $v \geq 0$; for this reason $V^+$ is sometimes called the \emph{cone of positive elements of $V$}.
\end{definition}

\begin{remark}
Although we have used the set $V^+$ to define a partial ordering, we
could just as easily have gone the other direction: If $\geq$ is a
partial ordering on $V$ that is translation invariant and invariant
under multiplication by non-negative reals, then the set $V^+ := \{ v
\in V : v \geq 0 \}$ is a set satisfying $(a), (b)$ and $(c)$ above, and consequently $(V,V^+)$ is an ordered vector space.
\end{remark}

\begin{remark}
There is some disagreement in the literature about the use of the term ``cone''.  While there are authors that conform with the definitions as we have given them (e.g., \cite{KN}, \cite{CE2}), there are other authors (e.g., \cite{KG}, \cite{Peres}, \cite{Sch}) who use the term ``cone'' to refer to a subset of $V$ satisfying $(a)$, $(b)$, and $(c)$ stated above, and use the term ``wedge'' for a subset satisfying only $(a)$ and $(b)$.  Since much of our work is motivated by operator systems, we have chosen to use the same terminology as \cite{CE2}.
\end{remark}

\begin{definition} \label{ord-unit-real-def}
If $(V,V^+)$ is an ordered real vector space, an element $e \in V$ is called an \emph{order unit} for $V$ if for each $v \in V$ there exists a real number $r > 0$ such that $re \geq v$.
\end{definition}

\begin{definition}
If $(V,V^+)$ is an ordered real vector space, then we say that $V^+$ is \emph{full} (or that $V^+$ is a \emph{full cone}) if $V = V^+ - V^+$.
\end{definition}

\begin{lemma} \label{basic-order-unit-facts}
If $(V,V^+)$ is an ordered real vector space with order unit $e$, then
\begin{itemize}
\item[(a)] $e \in V^+$;
\item[(b)] If $v \in V$ and a real number $r >0$ is chosen so that $re \geq v$, then $se \geq v$ for all $s \geq r$;
\item[(c)] If $v_1, \ldots, v_n \in V$, then there exists $r > 0$ such that $re \geq v_i$ for all $1 \leq i \leq n$; 
\item[(d)] $V^+$ is a full cone of $V$;
\item[(e)] If $v_1, \ldots, v_n \in V^+$ and $v_1 + \ldots + v_n = 0$, then $v_1= \ldots = v_n = 0$; and 
\item[(f)] Suppose $v_1, \ldots, v_n \in V^+$ and $a_1, \ldots, a_n \in [0, \infty)$.  If $a_1v_1 + \ldots + a_n v_n = 0$, then for each $1 \leq i \leq n$ either $a_i = 0$ or $v_i = 0$.
\end{itemize}
\end{lemma}

\begin{proof}
To see $(a)$ note that there exists $r > 0$ such that $re \geq -e$.  But then $re+e \in V^+$ and $e = (r+1)^{-1}(re+e) \in V^+$.

For $(b)$ we see that if $re \geq v$, then since $s \geq r$ we have that $s-r \geq 0$ and using $(a)$ it follows that $(s-r)e \geq 0$ and hence $(s-r)e + re \geq v +0$ and $se \geq v$.

To verify $(c)$ let $v_1, \ldots, v_n \in V$.  For each $1 \leq i \leq n$ choose a real number $r_i > 0$ such that $r_i e \geq v_i$.  If we let $r := \max \{ r_1, \ldots r_n \}$, then it follows from $(b)$ that $re \geq v_i$ for all $1 \leq i \leq n$.

To show $(d)$ let $v \in V$.  Using $(c)$ we may choose a real number $r >0$ such that $re \geq -v$ and $re \geq v$.  But then $re + v \in V^+$ and $re - v \in V^+$, and it follows that $v = (re +v)/2 - (re-v)/2 \in V^+ - V^+$.

We shall prove $(e)$ by induction.  For $n=1$ the claim holds trivially.  If $v_1 + \ldots + v_n = 0$, then $v_n = -v_1 - \ldots - v_{n-1}$, then $v_n \in V^+ \cap -V^+$ so $v_n = 0$.  Thus $v_1 + \ldots + v_{n-1} = 0$ and by the inductive hypothesis $v_1 = \ldots = v_{n-1} = 0$.

For $(f)$, notice that each $a_iv_i \in V^+$.  Thus Part~$(e)$ implies that $a_iv_i = 0$ for all $i$.  Hence either $a_i = 0$ or $v_i = 0$.
\end{proof}

\begin{definition} \label{Arch-ord-unit-real-def}
If $(V, V^+)$ is an ordered real vector space with an order unit $e$, then we say that $e$ is an \emph{Archimedean order unit} if whenever $v \in V$ with $re+v \geq 0$ for all real $r >0$, then $v \in V^+$.
\end{definition}

Note that having $e$ Archimedean is a way of ensuring that there are no ``non-positive infinitesimals'' in $V$.

\begin{lemma} \label{translated-Archimedean}
Let $(V, V^+)$ be an ordered real vector space with an Archimedean order unit $e$, and let $r_0 \in [0, \infty)$.  If $v \in V$ and $re + v \geq 0$ for all $r > r_0$, then $r_0e+v \geq 0$.
\end{lemma}

\begin{proof}
Since $re+v \geq 0$ for all $r > r_0$, we have that $se + (r_0e+v) \geq 0$ for all $s >0$, and thus $r_0e+v \geq 0$.
\end{proof}

\subsection{Positive $\R$-linear functionals and states} \label{subsec-positive-R-functionals}

\begin{definition}
If $(V,V^+)$ is an ordered real vector space, an $\R$-linear functional $f : V \to \R$ is called \emph{positive} if $f(V^+) \subseteq [0, \infty)$. 
\end{definition}

\begin{definition}
If $(V,V^+)$ is an ordered real vector space and $S \subseteq V$, we say that $S$ \emph{majorizes} $V^+$ if for each $v \in V^+$ there exists $w \in S$ such that $w \geq v$. 
\end{definition}

\begin{remark}
Note that if $S$ majorizes $V^+$ and $S \subseteq T$, then $T$ majorizes $V^+$.  In addition, if $e$ is an order unit for $(V,V^+)$, and if $E$ is a subspace of $V$ containing $e$, then $E$ majorizes $V^+$.  This is due to the fact that if $v \in V$, then there exists a real $r > 0$ such that $re \geq v$, and since $E$ is a subspace we have that $re \in E$.
\end{remark}

We wish to prove Theorem~\ref{Hahn-Banach-Thm}, which may be thought of as an analogue of the Hahn-Banach Theorem since it gives conditions under which we may extend a positive $\R$-linear functional on a subspace of $V$ to a positive $\R$-linear functional on all of $V$.  Before proving this theorem we will need the following lemma.

\begin{definition}
Suppose $(V,V^+)$ is an ordered real vector space and that $V^+$ is a
full cone for $V$.  Let $E$ be a subspace of $V$ that majorizes $V^+$,
and let $f : E \to \R$ be a positive $\R$-linear functional.  Given $h
\in V$ set $$L_h := \{ z \in E: z \leq h \}$$ and
$$U_h : = \{ z \in E : h \leq z \}.$$ Since $E$ is a subspace that majorizes $V^+$ and $V^+$ is a full cone, these
sets are nonempty. We define 
$$\ell_f(h) := \sup \{ f(z) : z \in L_h \},$$ and
 $$
u_f(h) := \inf \{ f(z) : z \in U_h \}.$$
\end{definition}

Note that if $z \in L_h$ and $w \in U_h$, then $z \leq w$.  Hence $\ell_f(h)
\le u_f(h).$ Also, if $h \in E$, then $\ell_f(h) =f(h) = u_f(h).$ 

\begin{lemma} \label{extend-pos-codim-1}
Suppose $(V,V^+)$ is an ordered real vector space and that $V^+$ a full cone for $V$.  Also suppose that $E$ is a subspace of $V$ that majorizes $V^+$, that $h \notin E$, and that
$f: E \to \R$ is a positive $\R$-linear functional. Let $W =
\{ ah+v : a \in \R \text{ and } v \in E \}$ be the real subspace
spanned by $E$ and $h.$ If $\gamma \in \R$ and $\ell_f(h) \le \gamma \le
u_f(h)$ and if we define $f_{\gamma} : W \to \R$ by
$$f_{\gamma}(ah+v) := a \gamma + f(v),$$ then
$f_{\gamma}$ is a positive $\R$-linear functional and
$f_{\gamma}|_E = f$. Moreover, if $g:W \to \R$ is a
positive $\R$-linear functional with $g |_E = f,$ then
$\ell_f(h) \le g(h) \le u_f(h).$
\end{lemma}

\begin{proof}
It is straightforward to verify that $f_{\gamma}$ is a
well-defined $\R$-linear functional with $f_{\gamma}|_E =
f$.  It remains to show that $f_{\gamma}$ is positive.

Suppose that $ah+v \in W$ with $ah+v \geq 0$.  We shall show that $f_{\gamma}(ah+v) \geq 0$ by considering three cases.

\noindent \textsc{Case I:} $a=0$.

In this case $v \geq 0$, and $f_{\gamma}(ah+v) := 0 \gamma + f(v) = f(v) \geq 0$ due to the fact that $f$ is positive.

\noindent \textsc{Case II:} $a > 0$.

Since $ah +v \geq 0,$ we have that $h \geq (-1/a)v,$ so that $(-1/a)v
\in L_h.$ Hence $(-1/a)f(v) \leq \ell_f(h) \leq \gamma$. Thus $0
\leq f(v) +a \gamma = f_{\gamma}(ah +v).$

\noindent \textsc{Case III:} $a < 0$.

In this case $-a > 0$ and $(-1/a)v \geq h.$ Thus $(-1/a)v \in U_h$,
and hence $\gamma \leq (-1/a)f(v)$.  It follows that $0 \leq
(-a)[(-1/a)f(v) - \gamma] = f_{\gamma}(v + ah).$

Since $f_{\gamma}(ah+v) \geq 0$ in all cases, we have that  $f_{\gamma}$ is positive.

Finally, if $g:W \to \R$ is a positive $\R$-linear functional with $g |_E = f$, then for any $z \in L_h$ and $w \in U_h$, we have $z \le h \le
w$, and hence $f(z) = g(z) \le g(h) \le g(w) = f(w).$ Taking the
supremum of this inequality over all $z \in L_h$ and the infimum over
all $w \in U_h$ yields $\ell_f(h) \le g(h) \le u_f(h)$, and the
proof is complete.

\end{proof}

The following theorem generalizes \cite[Corollary~2.1]{Kad}.

\begin{theorem}  \label{Hahn-Banach-Thm}
Suppose $(V,V^+)$ is an ordered real vector space and that $V^+$ is a full cone for $V$.  If $E$ is a subspace of $V$ that majorizes $V^+$, and if $f : E \to \R$ is a positive $\R$-linear functional on $E$, then there exists a positive $\R$-linear functional $\widetilde{f} : V \to \R$ such that $\widetilde{f}|_E = f$.
\end{theorem}

\begin{proof}
Let $\mathcal{C}$ be the collection of all pairs $(E', f')$ where $E'$ is a subspace of $V$ and $f'$ is a positive $\R$-linear functional on $E'$ with $f'|_{E'} = f$.  Define $(E_1, f_1) \leq (E_2, f_2)$ to mean $E_1 \subseteq E_2$ and $f_2 |_{E_1} = f_1$.  Then $\mathcal{C}$ is a partially ordered set.  If $S = \{ (E_\lambda, f_\lambda) : \lambda \in \Lambda \}$ is a chain in $\mathcal{C}$, then we may define $E_0 := \bigcup_{\lambda \in \Lambda} E_\lambda$, which is a subspace since $S$ is a chain, and we may define $f_0 : E_0 \to \R$ by $f_0(x) := f_\lambda (x)$ if $x \in E_\lambda$.  It is easily checked that $f_0$ is a well-defined positive $\R$-linear functional on $E_0$, and $(E_0,f_0)$ is an upper bound for $S$.  By Zorn's Lemma, $\mathcal{C}$ has a maximal element $(\widetilde{E}, \widetilde{f})$.

We shall now show that $\widetilde{E} = V$.  If $\widetilde{E} \neq V$, then since $V = V^+ - V^+$, there exists $p \in V^+ \setminus \widetilde{E}$.  If we let $W := \spn_\R \{ p \} \cup \widetilde{E}$, then $\widetilde{E}$ is a proper subspace of $W$.  Furthermore, since $E \subseteq \widetilde{E}$ and $E$ majorizes $V^+$, it follows that $\widetilde{E}$ majorizes $V^+$.  Thus Lemma~\ref{extend-pos-codim-1} implies that $\widetilde{f}$ extends to a positive $\R$-linear functional $f_W$ on $W$.  But then $(W,f_W)$ is an element of $\mathcal{C}$ that is strictly greater than $(\widetilde{E}, \widetilde{f})$.  This contradicts the maximality of $(\widetilde{E}, \widetilde{f})$, and consequently we must have that $\widetilde{E} = V$.
\end{proof}

\begin{corollary} \label{Hahn-Banach-Cor}
Let $(V, V^+)$ be an ordered real vector space with order unit $e$.  If $E$ is a subspace of $V$ containing $e$, then any positive $\R$-linear functional $f : E \to \R$ may be extended to a positive $\R$-linear functional $\widetilde{f} : V \to \R$ such that $\widetilde{f}|_E = f$.
\end{corollary}

\begin{proof}
Note that $V^+$ is a full cone for $V$ by Lemma~\ref{basic-order-unit-facts}(d).  In addition, $E$ majorizes $V^+$ since $E$ contains the order unit $e$.
\end{proof}

\begin{definition} Let $(V,V^+)$ be an ordered real vector space with
  order unit $e$. A positive $\R$-linear functional $f:V \to \R$ is
  called a {\em state} if $f(e) = 1$, and we call the set of all states
  on $V$ the {\em state space} of $V.$
\end{definition}

By the above corollary, the $\R$-linear functional $f(re) =r$ defined on
$E= \spn_{\R} \{e \}$ is positive and can be extended to all of $V$. Hence the state space of $V$ is always nonempty.

\begin{theorem} \label{interval-thm}
Let $(V,V^+)$ be an ordered real vector space with
  order unit $e.$ If $v \in V,$ then $\alpha : = \sup \{ r \in \R : re \leq v
  \} \leq \inf \{ s \in \R : v \le se \} = : \beta$ and for every real number
  $\gamma \in [ \alpha, \beta]$ there exists a state
  $f_{\gamma} : V \to \R$ with $f_{\gamma}(v) = \gamma.$
\end{theorem}
\begin{proof} 
We shall apply Lemma~\ref{extend-pos-codim-1} with $E := \{ re : r \in
\R \},$ $f:E \to \R$ given by $f(re) =r$, and $h:=v.$ Note that in this
case, $\alpha = \ell_f(v)$ and $\beta = u_f(v).$ Hence, by the lemma,
for each $\gamma \in [\alpha, \beta]$ we have that the $\R$-linear functional
$g_{\gamma}$ on $W = \{ re +tv: r,t \in \R \}$ given by
$g_{\gamma}(re+tv) = r + t \gamma$ is positive.

By Corollary~\ref{Hahn-Banach-Cor}, the positive $\R$-linear functional $g_{\gamma} : W \to \R$ can be extended to a positive $\R$-linear functional $f_{\gamma} : V \to \R$. We then have that
$f_{\gamma}(e) =1$ and $f_{\gamma}(v) = \gamma$, so that $f_{\gamma}$
is the desired state.
\end{proof}

\begin{remark} \label{interval-rem} Let $(V,V^+)$ be an ordered real vector space with order unit $e,$  let $v \in V$, and let $\alpha, \beta$ be as in Theorem~\ref{interval-thm}. Then it is readily seen that
$\{ f(v): f \text{ a state on } V \}$ is the closed interval $[\alpha, \beta].$
\end{remark}

\begin{proposition} \label{pos-R-funct-separating}
Let $(V, V^+)$ be an ordered real vector space with Archimedean order unit $e$.  If $v \in V$ and $f(v) = 0$ for all states $f : V \to \R$, then $v = 0$.
\end{proposition}

\begin{proof} By the above remark, if $f(v) =0$ for every state, then $\alpha = \beta = 0.$ Since $\alpha =0$ we have that $v \geq (-r)e$ for all $r>0$; that is, $re+v \ge 0$ for every $r>0.$  Thus by the Archimedean property we have that $v \in V^+.$  However, since $\beta = 0$, we also have that $re \geq v$ for every $r > 0,$ and so again by the Archimedean property $-v \in V^+.$  Thus $v=0$.
\end{proof}

\begin{proposition} \label{pos-for-states-implies-pos-prop}
Let $(V, V^+)$ be an ordered real vector space with Archimedean order unit $e$.  If $v \in V$ and $f(v) \geq 0$ for every state $f : V \to \R$, then $v \in V^+$.
\end{proposition}

\begin{proof}
Let  $\alpha : = \sup \{ r: re \leq v \}$.  Then by Theorem~\ref{interval-thm} there exists a state $f_\alpha : V \to \R$ such that $f_\alpha (v) = \alpha$.  By hypothesis we have that $\alpha \geq 0$.  Thus $re \leq v$ for all $r < 0$, which implies that $se + v \geq 0$ for all $s > 0$.  By the Archimedean property we have that $v \geq 0$.
\end{proof}

\subsection{The order seminorm} \label{seminorm-from-order-unit}

We now describe a natural seminorm on an ordered vector space with an order unit.  This seminorm will be a norm if the order unit is Archimedean, and it can be used to define a topology on the vector space.

\begin{definition} \label{order-seminorm-def}
Let $(V,V^+)$ be an ordered real vector space with order unit $e$.  For $v \in V$, let $$\| v \| := \inf \{ r \in \mathbb{R} : re + v \geq 0 \text{ and } re-v \geq 0 \}.$$  Note that $\| \cdot \|$ depends on the choice of the order unit $e$.  Also note that since $e$ is an order unit, the set $\{ r \in \mathbb{R} : re + v \geq 0 \text{ and } re-v \geq 0\}$ is nonempty and thus the infimum exists and is a real number.  We call $\| \cdot \|$ the \emph{order seminorm} on $V$ determined by $e$.  
\end{definition}

\begin{remark}
It is not hard to see that $\inf \{r: re+v \ge 0 \text{ and } re -v
\ge 0 \} = \max \{ |\alpha|, |\beta| \}$, where  $\alpha$ and $\beta$
 are defined as in Theorem~\ref{interval-thm}.  To see this, note that
 $re+v \ge 0$ and $re-v \ge 0$ if and only if $-re \le v \le +re$ and
 such an $r$ is necessarily non-negative.
 Thus, for any such $r$ we have $-r \le \alpha \le \beta \le +r$ and
 so $r \ge \max \{ |\alpha|, |\beta| \}.$ Thus $\inf \{ r: re+v \ge 0
 \text{ and } re -v \ge 0 \} \ge \max \{ |\alpha|, |\beta| \}.$
 Conversely, if $t> \max \{ |\alpha|, |\beta| \},$ then $-t < \alpha$ and
 $\beta < +t$, and hence $-te \le v \le +te$, which proves the other
 inequality.
In \cite[Lemma~2.3]{Kad} Kadison proves that $\|v\| = \max \{
|\alpha|, |\beta| \}.$ is a norm in the case that the order unit is
Archimedean. Our next result includes this fact.
\end{remark}

\begin{proposition} \label{seminorm}  \label{Arch-gives-norm-cor}
If $(V,V^+)$ is an ordered real vector space with order unit $e$, and $\| \cdot \|$ is as in Definition~\ref{order-seminorm-def}, then $\| \cdot \|$ is a seminorm on $V$ and for each $v \in V$, we have that $$\|v\| = \sup \{ |f(v)| \, : \, f  : V \to \R \text{ is a state} \}.$$
Moreover, when $e$ is an Archimedean order unit, $\| \cdot \|$ is a norm.\end{proposition}

\begin{proof}
If $v \in V$, and if $\alpha$ and $\beta$ are defined as in Theorem~\ref{interval-thm}, then by Remark~\ref{interval-rem}, we have that $\sup \{ |f(v)|: f \text{ is a state on } V \} = \sup \{ |\gamma|: \alpha \le \gamma \le \beta \} = \max \{ |\alpha|, |\beta| \} = \|v\|.$  Thus we have $\|tv\| = \sup \{ |f(tv)| : f \text{ is a state on } V \} = \sup
\{|t||f(v)| : f \text{ is a state on } V \} = |t| \|v\|.$ For $v,w \in
V,$ we have that $\|v+w\| = \sup \{ |f(v+w)| : f \text{ is a state on
} V \} \le \sup \{ |f(v)| + |f(w)| : f \text{ is a state on } V \} \le
\|v\| + \|w\|.$ Hence $\| \cdot \|$ is a seminorm.

When $e$ is an Archimedean order unit,
Proposition~\ref{pos-R-funct-separating} shows that $\|v\|=0$
implies $v=0$, and so $\| \cdot\|$ is a norm.
\end{proof}

\begin{remark} 
Since the order seminorm is in fact a norm when $e$ is an Archimedean order unit, we shall henceforth refer to it as the {\em order norm} in this context.
\end{remark}

\begin{remark}
When $(V,V^+)$ is an ordered vector space with order unit $e$, the
above result shows that if $e$ is Archimedean, then the order seminorm $\| \cdot \|$ is a norm.  However, the converse does not necessarily hold.  For example, if $V = \C^2$ and we set $V^+ = \{ (x,y) : x > 0 \text{ and } y >0 \} \cup \{ 0 \}$, then $e = (1,1)$ is an order unit for $(V,V^+)$ and one can check that the order seminorm $\| \cdot \|$ is a norm.  However, $e$ is not Archimedean: one has $re + (1,0) \in V^+$ for all $r >0$, but $(1,0) \notin V^+$.
\end{remark}

\begin{definition} \label{order-seminorm-top-def}
Let $(V,V^+)$ be an ordered real vector space with order unit $e$.  The \emph{order topology} on $V$ is the topology induced by the order seminorm $\| \cdot \|$; i.e., the topology with a basis consisting of balls $B_{\epsilon}(v) := \{ w \in V : \|w-v \| < \epsilon\}$ for $v \in V$ and $\epsilon >0$.  Note that since $\| \cdot \|$ is not necessarily a norm, this topology is not necessarily Hausdorff.
\end{definition}

Note that if $V$ and $W$ are two spaces endowed with seminorms, then a linear map from $V$ into $W$ is continuous if and only if it is bounded.

\begin{proposition} \label{positive-R-functionals-cts}
Let $(V,V^+)$ be an ordered real vector space with an order unit $e$, and let $\| \cdot \|$ be the order seminorm on $V$ determined by $e$. If $f : V \to \R$ is a positive $\R$-linear functional, then $f$ is continuous with respect to the topology induced by $\| \cdot \|$, and $\| f \| = f(e)$.  In particular, $$|f(v)| \leq f(e) \| v \| \quad \text{ for all $v \in V$.}$$
\end{proposition}

\begin{proof}
Let $v \in V$.  If $r > \|v\|,$ then we have that $-re \leq v \leq +re$.  Since $f$ is positive and $\R$-linear, it follows that $-f(re) \leq f(v) \leq f(re)$, and thus $|f(v)| \leq f(re)$.  Using the $\R$-linearity, we then have $$|f(v)| \leq f(e) \| v \|.$$  It follows that $f$ is bounded, and hence continuous for the topology coming from the norm $\| \cdot \|$.  Furthermore, the above inequality shows that $\| f \| \leq f(e)$.  Because $\| e \| = 1$, we have that $\| f \| = f(e)$.
\end{proof}

We obtain a partial converse to the above proposition, which shows that the positive $\R$-linear functionals on an Archimedean ordered space are precisely the $\R$-linear functionals that are continuous in the order norm topology with $\| f \| = f(e)$.

\begin{proposition} \label{certain-R-functionals-cts-implies-pos}
Let $(V,V^+)$ be an ordered real vector space with order unit $e$, and let $\| \cdot \|$ be the order seminorm on $V$ determined by $e$. If $f : V \to \R$ is an $\R$-linear functional that is continuous with respect to the order norm topology, and if $\| f \| = f(e)$, then $f$ is positive.
\end{proposition}

\begin{proof}
Let $v \in V^+$.  For any $r > \|v\|$ we have that $0 \leq v \leq re$, and hence $0 \leq re - v \leq re$.  It follows that $\left\| (re - v)  \right\| \leq r$.  Since $f$ is continuous with $\| f \| = f(e)$, we have $$| f( re - v ) | \leq \| f \| \left\| (re - v)  \right\| \leq f(e) r $$ and hence $f( r e - v )  \leq rf(e) $ and $rf(e) - f(v) \leq rf(e)$, so that $0 \le f(v)$.  Thus $f$ is positive.
\end{proof}

We now characterize the order seminorm.

\begin{theorem} \label{norm-char} Let $(V,V^+)$ be an ordered real vector space with order unit $e$.  Then the order seminorm $\| \cdot \|$ is the unique seminorm on $V$ satisfying the following three conditions.
\begin{enumerate}
\item $\| e \| = 1$,
\item if $-w \le v \le w$, then $\| v \| \leq \| w \|$, and
\item if $f:V \to \R$ is a state, then $|f(v)| \leq  \|v\| $.
\end{enumerate}
\end{theorem}
\begin{proof}  First we show that the order seminorm satisfies these three properties. Let $\| \cdot \|$ be the order seminorm on $V$.  Since $re + e \geq 0$ if and only if $r \geq -1$ and $re - e \geq 0$ if and only if $r \geq 1$, we have that $\| e \| = 1$.  

To see $(2)$ let $r> \|w\|,$ then $-re \le w \le re$ and hence, $-re \le -w \le v \le w \le re$.  Thus $\|v\| \le r$, and hence $\|v\| \leq \inf \{r: r> \|w\| \} = \|w\|.$  Also, $(3)$ is a consequence of Proposition~\ref{seminorm}.

Furthermore, let $||| \cdot |||$ be a seminorm on $V$ satisfying the above three conditions.  If $v \in V$ and $r > \|v\|,$ we have that $-re \leq v \leq +re$.  Using Conditions (1) and (2) for $||| \cdot |||$ we have that $||| v ||| \leq ||| re ||| = r ||| e ||| = r$ and it follows that $|||v||| \le \|v\|.$  Conversely, Condition~(3) shows that $$\sup \{ |f(v)|: f \text{ is a state on } V \} \leq ||| v |||$$ and thus by Proposition~\ref{seminorm} we have that $\| v \| \leq ||| v |||$.  Hence $||| v ||| = \| v \|$.
\end{proof}

\begin{theorem} \label{Arch-equiv-positive-closed-prop}
Let $(V,V^+)$ be an ordered real vector space with order unit $e$, and let $\| \cdot \|$ be the order seminorm on $V$ determined by $e$.  Then the following are equivalent:
\begin{enumerate}
\item $e$ is Archimedean,
\item $V^+$ is a closed subset of $V$ in the order topology induced by $\| \cdot \|,$
\item $-\|v\| e \leq v \leq \|v\| e$ for all $v \in V$.
\end{enumerate}
\end{theorem}

\begin{proof} 

\noindent $(1) \implies (2)$.  Let $e$ be Archimedean.  Suppose that $v \in V$  and that $v$ is a limit point of $V^+$.  Then for any $r >0$, there is an element $v_r \in B_r (v) \cap V^+$.  But then $\| v - v_r \| < r$, and $re + v - v_r \geq 0$.  Since $v_r \geq 0$ this implies that $re + v \geq 0$.  Because $e$ is Archimedean, it follows that $v \geq 0$ and $v \in V^+$. Thus $V^+$ is closed.

\noindent $(2) \implies (1)$.  Suppose that $V^+$ is closed in the topology induced by $\| \cdot \|$.  Let $v \in V$ with the property that $re + v \geq 0$ for all $r > 0$.  Since $\| re \| = r \|e \| = r$, we see that for any $r > 0$ we have $re + v \in B_{2r}(v) \cap V^+$.  Thus $v$ is a limit point of $V^+$, and it follows that $v \in V^+$ and $v \geq 0$.

\noindent $(1) \implies (3)$.  Let $v \in V$.  Since $\| v \| := \inf \{ r : re + v \geq 0 \text{ and } re - v \geq 0 \}$, we see that $re + v \geq 0$ and $re - v \geq 0$ for all $r > \| v \|$.  It follows from Lemma~\ref{translated-Archimedean} that $\|v \|e + v \geq 0$ and $\| v \| e - v \geq 0$.  Hence $ -\| v \| e \leq v \leq \|v\| e$.

\noindent $(3) \implies (1)$.  Let $v \in V$ with $re+v \in V^+$ for all $r > 0$.  Consider $\| v \| e - v$.  By hypothesis $\| v \| e - v \geq 0$, and hence $re + ( \| v \| e - v) \geq 0$ for all $r >0$.  In addition, since $re+v \geq 0$ for all $r >0$, it follows that $(r - \| v \|)e + v \geq 0$ for all $r > \| v \|$, and $re - (\| v \|e-v) \geq 0$ for all $r >0$.  The inequalities in the previous two sentences imply that $\| \, \| v \| e - v \, \| \leq \| v \|$.  By (3) we have that $\| v \| e - v \leq \| \, \| v \| e - v \, \| e$, and hence $\| v \| e - v \leq \| v \| e$.  Thus $v \geq 0$, and $v \in V^+$ so that $e$ is Archimedean.
\end{proof}

\begin{remark}
We now turn our attention to morphisms between ordered vector spaces, and discuss categories of ordered vector spaces and Archimedean ordered vector spaces.  If $(V,V^+)$ is an ordered vector space with order unit $e$ and $(W,W^+)$ is an ordered vector space with order unit $e'$, then a linear map $\phi : V \to W$ is called \emph{positive} if $v \in V^+$ implies $\phi(v) \in W^+$.  In addition, $\phi$ is called \emph{unital} if $\phi(e) = e'$.  If $\phi : V \to W$ is a unital positive linear map that is a  bijection, it is not necessarily the case that $\phi^{-1} : W \to V$ is positive.  We call a linear map $\phi : V \to W$ an \emph{order isomorphism} when $v \in V^+$ if and only if $\phi(v) \in W^+$.  Note that if $\phi : V \to W$ is an order isomorphism, then $\phi^{-1}$ is a positive linear map.

We wish to consider the category $\mathcal{O}$ whose objects are ordered vector spaces with an order unit and whose morphisms are unital positive linear maps.  The isomorphisms in this category are unital order isomorphisms.  Furthermore, the class of ordered vector spaces with Archimedean order units forms a full subcategory $\mathcal{O}_\textrm{Arch}$ of $\mathcal{O}$.  

We would like to consider a functorial method for taking an ordered vector space and turning it into an Archimedean ordered vector space.  We will call this process \emph{Archimedianization}.  Furthermore, we would like to consider quotients of ordered vector spaces by order ideals.  This is straightforward for ordered vector spaces with an order unit: given an ordered space $V$ with order unit $e$ and given an order ideal $J$, the quotient vector space $V/J$ is naturally an ordered vector space with order unit $e+J$.  This gives a notion of quotients in the category $\mathcal{O}$.  However, this does not work for $\mathcal{O}_\textrm{Arch}$ because the quotient of an Archimedean ordered vector space by an order ideal is not necessarily Archimedean.  In order to have a notion of quotient in the category $\mathcal{O}_\textrm{Arch}$, we will make use of the Archimedeanization process and define the \emph{Archimedean quotient} of an Archimedean ordered vector space $V$ by an order ideal $J$ to be the Archimedeanization of the quotient $V/J$.
\end{remark}

\subsection{The Archimedeanization of an ordered real vector space} \label{real-Arch-subsec}

Suppose that $(V,V^+)$ is an ordered real vector space with an order unit $e$ that is not Archimedean.  There are two obstructions to having $e$ be Archimedean: (1) the order seminorm $\| \cdot \|$ may not be a norm (see Corollary~\ref{Arch-gives-norm-cor}), and (2) there may be non-positive infinitesimals in $V$.  We shall show that by quotienting out by the vectors of norm $0$ and then enlarging the positive elements to prevent non-positive infinitesimals, we may create an ordered vector space in which the equivalence class of $e$ is an Archimedean order unit.  In this way we may ``Archimedeanize'' $V$.

\begin{definition}
Let  $(V,V^+)$ be an ordered real vector space with an order unit $e$.  Define $D := \{ v \in V : re + v \in V^+ \text{ for all $r >0$} \}$.  Also define $N := D \cap - D$.
\end{definition}

\begin{remark} \label{D-N-real-remark}
It is straightforward to verify that $D$ is a cone with $V^+ \subseteq D$ and that $N$ is a real subspace of $V$.
\end{remark}

\begin{proposition}
Let $(V,V^+)$ be an ordered real vector space with an order unit $e$. 
Then $D$ is equal to the closure of $V^+$ in the order topology and
$N = \{ v \in V : \| v \| = 0 \} = \bigcap_{f: V \to \R \atop \text{is a state}} \ker f .$
\end{proposition}
\begin{proof}

If $v \in D,$ then $re+v \in V^+$ with $\|(re+v) -v\| = r$ for all $r > 0$ and so $D$ is contained in the closure of $V^+$.  Conversely, assume that $v$ is in the closure of $V^+$ so that there exists  sequence $\{ v_n \}_{n=1}^\infty \subseteq V^+$ with $\|v_n -v \| \to 0$. For each $n \in \N$ choose $r_n > \|v_n -v\|$ with $r_n \to 0$.  Then $r_ne \pm (v_n -v) \in V^+,$  and hence $r_ne +v - v_n \in V^+$, which implies that $r_ne +v \in V^+$ for all $n$. Since $r_n \to 0$ it follows that $re +v \in V^+$ for all $r > 0.$  Thus the closure of $V^+$ is contained in $D$, and so the two sets are equal.

The equality $\{ v \in V : \| v \| = 0 \} = \bigcap_{f: V \to \R \atop \text{is a state}} \ker f$ follows from Proposition~\ref{seminorm}, so we prove the first equality.  If $v \in N$, then $re \pm v \in V^+$ for all $r > 0.$ Hence, $-re \leq v \leq re$ for all $r > 0.$ By Theorem~\ref{norm-char} we have $\|v\| \leq \|re\|=r$ for all $r >0,$ and so $\|v\| =0$.  Conversely, if $\|v\|=0,$ then $re \pm v \in V^+$ for all $r > 0.$ This implies that $\pm v \in D,$ or equivalently that $v \in D \cap -D = N.$ 

\end{proof}

\begin{theorem} \label{V/N-is-Arch-prop}
Let  $(V,V^+)$ be an ordered real vector space with an order unit $e$.  Let $N := D \cap -D$, and consider $V / N$ with $$(V / N )^+ := D+N = \{ v + N : v \in D \}.$$  Then $(V / N, (V / N)^+)$ is an ordered vector space and $e+N$ is an Archimedean order unit for this space.
\end{theorem}

\begin{proof}
Since $D$ is a cone, it follows that $(V / N )^+ := D+N$ is a cone.  In addition, if $x + N \in (V / N )^+ \cap -(V / N )^+$, then $x+N = d_1 +N$ and $x+N = -d_2 +N$ for some $d_1, d_2 \in D$.  Hence $x-d_1 \in N \subseteq D$ and since $D$ is a cone, we have that $x \in D$.  Likewise $x+d_2 \in N \subseteq -D$, and since $D$ is a cone we have that $x \in -D$.  Hence $x \in N:= D \cap -D$, and $x + N = 0 + N$.  Thus $(V/N)^+ \cap -(V/N)^+ = \{ 0 \}$.

In addition, given any $v + N \in V/N$, since $e$ is an order unit for $(V,V^+)$ there exists $r > 0$ such that $re + v \in V^+$.  Because $V^+ \subseteq D$, it follows that $r(e+N) + (v+N) = (re+v) +N \in D+N = (V / N )^+$.  Hence $e+N$ is an order unit for $V/N$.

Finally, suppose that $v + N \in V/N$ and that $r (e+N) + (v+N) \in (V / N )^+$ for all $r > 0$.  Then $(re+v) + N \in D+N$ and $re+v \in D$ for all $r > 0$.  Choose any $r_0 > 0$.  Then $(r_0/2) e +v \in D$, and by the definition of $D$, we have that $(r_0/2)e + [(r_0/2)e +v] \in V^+$.  It follows that $r_0e + v \in V^+$ for all $r_0 >0$.  By the definition of $D$, we have that $v \in D$.  Thus $v+N \in (V/N)^+$, and $e+N$ is an Archimedean order unit.
\end{proof}

\begin{corollary}
Let $(V,V^+)$ be an ordered vector space with an order unit $e$. If the order seminorm is a norm on $V$, then $(V,D)$ is an ordered vector space with $e$ an Archimedean order unit.
\end{corollary}
\begin{proof} In this case we have that $N=  \{ 0 \}.$
\end{proof}

\begin{definition}
Let  $(V,V^+)$ be an ordered real vector space with an order unit $e$.  Let $D := \{ v \in V : re + v \in V^+ \text{ for all $r >0$} \}$ and $N := D \cap -D$.  We define $V_\textnormal{Arch}$ to be the ordered vector space $(V/N, (V/N)^+)$ with the Archimedean order unit $e+N$.  We call $V_\textnormal{Arch}$ the \emph{Archimedeanization} of $V$.
\end{definition}

The following result describes a universal property that characterizes the Archimedeanization.

\begin{theorem} \label{real-Arch-charact-thm}
Let  $(V,V^+)$ be an ordered real vector space with an order unit $e$, and let $V_\textnormal{Arch}$ be the Archimedeanization of $V$.  Then there exists a unital surjective positive linear map $q : V \to V_\textnormal{Arch}$ with the property that whenever $(W,W^+)$ is an ordered vector space with Archimedean order unit $e'$, and $\phi : V \to W$ is a unital positive linear map, then there exists a unique positive linear map $\widetilde{\phi} : V_\textnormal{Arch} \to W$ with $\phi = \widetilde{\phi} \circ q$.
\begin{equation*}
\xymatrix{  V_\textnormal{Arch} \ar@{-->}[rd]^{\widetilde{\phi}} & \\ V \ar[u]^q  \ar[r]^<>(.35)\phi & W}
\end{equation*}
In addition, this property characterizes $V_\textnormal{Arch}$:  If $V'$ is any ordered vector space with an Archimedean order unit and $q' : V \to V'$ is a unital surjective positive linear map with the above property, then $V'$ is isomorphic to $V_\textnormal{Arch}$ via a unital order isomorphism $\psi : V_\textnormal{Arch} \to V'$ with $\psi \circ q = q'$.
\end{theorem}

\begin{proof}
Recall that $V_\textnormal{Arch} = V/N$ with positive cone $D+N$ and Archimedean order unit $e+N$.  Define $q : V \to V/N$ to be the quotient map.  Then $q$ is linear and unital, and since $V^+ \subseteq D$ it follows that $q$ is positive.  If $(W,W^+)$ is an ordered vector space with Archimedean order unit $e'$, and if $\phi : V \to W$ is a unital positive linear map, then for any $v \in N$ we have that $re +v \in V^+$ and $re-v \in V^+$ for all $r >0$.  Applying $\phi$ we obtain that $re'+\phi(v) \in W^+$ and $re' -\phi(v) \in W^+$ for all $r >0$.  Since $e'$ is an Archimedean order unit for $W$, it follows that $\phi(v) = 0$.  Thus $N \subseteq \ker \phi$.  Hence the map $\widetilde{\phi} : V/N \to W$ defined by $\widetilde{\phi} (v+N) = \phi(v)$ is well defined and makes the above diagram commute.  Furthermore, if $v+N \in (V/N)^+ = D + N$, then we may assume $v \in D$ and $re+v \in V^+$ for all $r >0$.  Applying $\phi$ gives that $re'+\phi(v) \in W^+$ for all $r>0$ and since $e'$ is Archimedean, it follows that $\phi(v) \in W^+$.  Thus $\widetilde{\phi}(v+N) = \phi(v) \in W^+$ and $\widetilde{\phi}$ is a positive map.  Finally, to see that $\widetilde{\phi}$ is unique, simply note that any $\psi : V_\textnormal{Arch} \to W$ that makes the above diagram commute would have $\psi (v + N) = \psi ( q (v)) = \phi (v) = \widetilde{\phi} (q (v)) = \widetilde{\phi} (v+N)$ so that $\psi = \widetilde{\phi}$.

The fact that $V_\textnormal{Arch}$ is characterized up to unital order isomorphism by the universal property follows from a standard diagram chase.
\end{proof}

\begin{remark}
The Archimedeanization may be viewed as a functor from the category
$\mathcal{O}$ of ordered vector spaces with an order unit to the
subcategory $\mathcal{O}_\textrm{Arch}$ of Archimedean ordered vector
spaces.  This functor takes an object $V$ to $V_\textrm{Arch}$, and if
$\phi : V \to W$ is a unital positive linear map, then one can see
that $\phi (D_V) \subseteq D_W$ and $\phi (N_V) \subseteq N_W$ so that
the induced map $\widetilde{\phi} : V_\textrm{Arch} \to
W_\textrm{Arch}$ is well defined and positive. Note that this functor
fixes the subcategory $\mathcal{O}_\textrm{Arch}$; i.e., it is a
projection onto the subcategory. 
\end{remark}

\subsection{Quotients of ordered real vector spaces} \label{Real-quotients-subsec}

We are now prepared to discuss quotients of ordered vector spaces.  Note that if $\phi$ is a positive map and $0 \le q \le p$ with $\phi(p) =0$, then necessarily $\phi(q)=0.$ This observation motivates the following definition.

\begin{definition} If $(V,V^+)$ is an ordered vector space, then a subspace $J \subseteq V$ is called an \emph{order ideal} provided that $p \in J$ and $0 \leq q \leq p$ implies that $q \in J$.
\end{definition}

\begin{proposition} \label{order-quotients-prop}
Let $(V,V^+)$ be an ordered vector space with order unit $e$, and let $J \subseteq V$ be an order ideal. Then $(V/J, V^++J)$ is an ordered vector space with order unit $e+J.$
\end{proposition}

\begin{proof}  It is easily seen that $V^++J$ is a cone, and that $e+J$ is an order unit. It remains to see that $(V^++J)\cap (-V^++J) = \{ 0 +J \}$. To this end, assume that there exist $p,q \in V^+$ such that $p+J = -q+J$. Then there exists $j \in J$ such that $p= -q +j.$  This implies that $j = p+q \in V^+,$ and hence $0 \leq p \leq j$ so that $p \in J$.  Thus $p+J = 0 +J$, and the proof is complete.
\end{proof}

\begin{definition} Let $(V,V^+)$ be an ordered vector space with order unit $e$, and let $J \subseteq V$ be an order ideal. We call the ordered vector space $(V/J, V^++J)$ with order unit $e+J$, the \emph{quotient of $V$ by $J$}.   If, in addition, $e$ is an Archimedean order unit, then we call the Archimedeanization of $(V/N, V^++J)$ the \emph{Archimedean quotient of $V$ by $J$}.
\end{definition}

\begin{remark}
There are examples (see, for instance, the example described in the remark on the bottom of \cite[p.67]{Alf}) showing that the quotient $V/J$ of an Archimedean ordered vector space $V$ by a closed order ideal $J$ need not be Archimedean.  This is why we need to consider the Archimedean quotient.
\end{remark}

By definition, the Archimedean quotient of $V$ by an order ideal $J$ is a quotient of $V/J$, and therefore a quotient of a quotient of $V$.  The following proposition shows how we may view this as a single quotient of $V$.

\begin{proposition} \label{Arch-one-quotient-prop}
Let $(V,V^+)$ be an ordered vector space with order unit $e$, and let $J$ be an order ideal of $V$.  If we define \begin{align*}
N_J := \{ v \in V : \text{for all } &r > 0 \text{ there exists } j,k \in J  \text{ such that } \\
&  j+re+v \in V^+ \text{ and } k + re - v \in V^+ \},
\end{align*}
then $N_J$ is an order ideal of $V$.  Furthermore, if we let 
\begin{align*}
(V/N_J)^+ := \{ v + N_J : \text{for all } &r > 0 \text{ there exists } j \in J \text{ such that } \\
& j + re + v \in V^+  \}
\end{align*}
then $(V/N_J, (V/N_J)^+ )$ is an ordered vector space with Archimedean order unit $e + N_J$, and this space is unitally order isomorphic to the Archimedean quotient of $V$ by $J$.
\end{proposition}

\begin{proof}
It is straightforward to verify that $N_J$ is a subspace of $V$.  To see that $N_J$ is an order ideal suppose $0 \leq q \leq p$ with $p \in N_J$.  Let $r > 0$.  Since $p \in N_J$, there exists $j,k \in J$ such that $j + re + p \in V^+$ and $k+re-p \in V^+$.  Since $p,q \in V^+$ and $p-q \in V^+$, this implies that $k + re - q = (k + re - p) + (p-q) \in V^+$.  In addition, $k + re + q =  (k+re-q) + 2q \in V^+$.  Thus $q \in N_J$.

The fact that $(V/N_J)^+$ is a cone is straightforward.  In addition, we see that if $v + N_J \in (V/N_J)^+ \cap -(V/N_J)^+$, then it follows that $v \in N_J$.  Hence $(V/N_J)^+ \cap -(V/N_J)^+ = \{ 0 + N_J \}$.  Thus $(V/N_J, (V/N_J)^+ )$ is an ordered vector space.

To see that $e+N_J$ is an order unit, note that $V^+ + N_J \subseteq (V/N_J)^+$.  Given $v + N_J \in V / N_J$, choose $r > 0$ such that $re - v \in V^+$.  Then $$r(e+N_J) - (v+N_J) = (re-v)+N_J  \in V^+ + N_J \subseteq (V/N_J)^+$$ so that $r(e+N_J) \geq v + N_J$.

To see that $e+N_J$ is an Archimedean order unit, suppose that $$r(e+N_J) + (v +N_J) \in (V/N_J)^+ \text{ for all $r > 0$.}$$  Then for any $r > 0$ we have that $((r/2)e + v)+N_J \in (V/N_J)^+$.  Thus there exists $j \in J$ such that $j + (r/2)e + (r/2e + v) \in V^+$.  Hence $j + re + v \in V^+$, and since $r>0$ was arbitrary we have that $v + N_J \in (V/N_J)^+$.

Finally, let $D := \{ v + J \in V/J : (re + v)+ J \in V^+ + J \text{ for all $r >0$}\}$ and $N := D \cap -D$.  Then, by definition, the Archimedean quotient of $V$ by $J$ is the ordered vector space $( (V/J)/N, D + N)$ with Archimedean order unit $(e+J) + N$.  Define $\psi : (V/J)/N \to V / N_J$ by $\psi ((v+J) +N) = v + N_J$.  We see that $\psi$ is well defined because if $(v+J) + N = (w +J) +N$, then $(v-w) + J \in N$ so that for all $r >0$ we have $[re \pm (v-w)] +J \in V^+ + J$, and hence there exists $j \in J$ such that $j + re \pm (v-w) \in V^+$, and $v + N_J = w + N_J$.  Furthermore, $\psi$ is clearly unital, linear, and surjective.  To see that $\psi$ is injective, let $\psi ((v+J) +N) = 0$.  Then $v \in N_J$, so for all $r > 0$ there exists $j,k \in J$ such that $j + re + v \in V^+$ and $k + re - v \in V^+$.  Thus $(re \pm v) + J \in V^+ + J$ so $v +J \in D \cap -D = N$, and $(v+J)+ N = 0$.  Moreover, we have that 
\begin{align*}
(v+J)+N \in D + N &\iff v+J \in D \\
&\iff (re + v)+J \in V^+ + J \text{ for all $r > 0$} \\
&\iff \text{$\forall$ $r > 0$ $\exists$ $j \in J$ such that $j + re + v \in V^+$} \\
&\iff v + N_J \in (V / N_J)^+.
\end{align*}
Thus $\psi$ is an order isomorphism.
\end{proof}

\begin{theorem} \label{first-iso-real-thm}
Let $(V,V^+)$ be an ordered vector space with Archimedean order unit $e$, and let $(W,W^+)$ be an ordered vector space with Archimedean order unit $e'$.  If $\phi : V \to W$ is a unital positive linear map, then $\ker \phi$ is an order ideal and the Archimedean quotient of $V$ by $\ker \phi$ is unitally order isomorphic to $V / \ker \phi$ with positive cone 
$$(V / \ker \phi)^+ = \{ v + \ker \phi : \text{$\forall$ $r >0$ $\exists$ $j \in \ker \phi$ such that $j + re + v \in V^+$ } \} $$ and Archimedean order unit $e + \ker \phi$.  In addition, the map $\widetilde{\phi} : V / \ker \phi \to W$ given by $\widetilde{\phi} (v + \ker \phi) = \phi (v)$ is a unital positive linear map.  Moreover, if $\phi (V^+) = W^+ \cap \im \phi$, then $\widetilde{\phi}$ is an order isomorphism from $V / \ker \phi$ onto $\im \phi$.
\end{theorem}

\begin{proof}
To see that $\ker \phi$ is an order ideal, suppose that $0 \leq q \leq p$ and $p \in \ker \phi$.  Since $\phi$ is positive we have $0 \leq \phi(q) \leq \phi(p) = 0$, so that $\phi(q) = 0$ and $q \in \ker \phi$.

In addition, if we let $J := \ker \phi$, then we shall show that $J = N_J$.  We trivially have that $J \subseteq N_J$.  For the reverse inclusion suppose that $v \in N_J$.  Then for all $r > 0$ there exists $j,k \in J$ such that $j + re + v \in V^+$ and $k + re - v \in V^+$.  Because $\phi$ is positive and unital, $re' + \phi(v) = \phi(j + re + v) \in W^+$ and $re' - \phi(v) = \phi (k + re - v) \in W^+$.  Thus the fact that $e'$ is an Archimedean order unit implies that $\phi(v) = 0$ and $v \in \ker \phi = J$.

It follows from Proposition~\ref{Arch-one-quotient-prop} that the Archimedean quotient of $V$ by $\ker \phi$ is isomorphic to the quotient $V / \ker \phi$ with positive cone $$(V / \ker \phi)^+ = \{ v + \ker \phi : \text{$\forall$ $r >0$ $\exists$ $j \in \ker \phi$ such that $j + re + v \in V^+$ } \} $$ and Archimedean order unit $e + \ker \phi$.  It then follows that $v + J \in (V / \ker \phi)^+$ implies that for all $r > 0$ there is $j \in \ker \phi$ such that $j + re + v \in V^+$.  By the positivity of $\phi$ we have $\phi (j + re + v) \in W^+$ and $re'+\phi(v) \in W^+$ for all $r  >0$.  Since $e'$ is an Archimedean order unit for $W$, we have that $\phi(v) \in W^+$ and $\widetilde{\phi}$ is a positive map.  

Moreover, if $\phi (V^+) = W^+ \cap \im \phi$, then $\widetilde{\phi} (v) \in W^+$ implies $\phi(v) \in W^+ \cap \im \phi$, so that $\phi (v) \in \phi (V^+)$.  Hence $\phi(v) = \phi(p)$ for some $p \in V^+$.  Thus $v + j = p$ for some $j \in \ker \phi$, and it follows that $j + re + v \in V^+$ for all $r > 0$.  Thus $v + \ker \phi \in (V / \ker \phi)^+$, and $\widetilde{\phi}$ is an order isomorphism onto $\im \phi$.
\end{proof}

\begin{remark}
Theorem~\ref{first-iso-real-thm} may be viewed as a ``First Isomorphism Theorem" for Archimedean ordered vector spaces using the Archimedean quotient.  Note, however, that $\widetilde{\phi}$ need not be an order isomorphism without the condition $\phi (V^+) = W^+ \cap \im \phi$ satisfied.
\end{remark}

\section{Ordered $*$-vector spaces} \label{Complex-Vec-Spaces}

In this section we will consider complex vector spaces with a $*$-operation, and describe what it means for these spaces to be ordered.  We also develop analogues of the results in \S \ref{Real-Vec-Spaces} for these ordered complex vector spaces.

\begin{definition}
A \emph{$*$-vector space} consists of a complex vector space $V$ together with a map $* : V \to V$ that is involutive (i.e., $(v^*)^* = v$ for all $v \in V$) and conjugate linear (i.e., $(\lambda v + w)^* = \overline{\lambda} v^* + w^*$ for all $\lambda \in \C$ and $v,w \in V$).  If $V$ is a $*$-vector space, then we let $V_h := \{ x \in V : x^* = x \}$ and we call the elements of $V_h$ the \emph{hermitian elements} of $V$.
\end{definition}

It is easy to see that $V_h$ is a real subspace of the complex vector space $V$.  Also note that any $v \in V$ may be written uniquely as $v = x + iy$ with $x, y \in V_h$.  In fact, $v = x+iy$ with $x,y \in V_h$ if and only if $x = (v + v^*)/2$ and $y = (v-v^*)/2i$.  We call $x$ and $y$ the \emph{real part} and \emph{imaginary part} of $v$, respectively, and we write $$\operatorname{Re} (v) := \frac{v+v^*}{2} \qquad \operatorname{Im} (v) := \frac{v-v^*}{2i}.$$ 

Note that we also have $V \cong V_h \oplus i V_h$ as real vector spaces.

\begin{definition}
If $V$ is a $*$-vector space, we say that $(V, V^+)$ is an \emph{ordered $*$-vector space} if $V^+ \subseteq V_h$ and the following two conditions hold:
\begin{enumerate}
\item $V^+$ is a cone of $V_h$ (i.e., $V^+ + V^+ \subseteq V^+$ and $aV^+ \subseteq V^+$ for all $a \in [0, \infty)$).
\item $V^+ \cap - V^+ = \{ 0 \}$.
\end{enumerate}
In this case for $v, w \in V_h$ we write $v \geq w$ to mean that $v-w \in V^+$.
\end{definition}

\begin{definition}
If $(V, V^+)$ is an ordered $*$-vector space, then an element $e \in V_h$ is an \emph{order unit} if for all $x \in V_h$ there exists a real number $r > 0$ such that $re \geq x$.  If $e$ is an order unit, we say that $e$ is \emph{Archimedean} if whenever $x \in V_h$ and $re+x \geq 0$ for all $r \in (0,\infty)$, then $x \geq 0$.
\end{definition}

\begin{remark}
Note that $(V, V^+)$ is an ordered $*$-vector space if and only if $(V_h, V^+)$ is an ordered real vector space in the sense of Definition~\ref{real-ord-v-s-def}.  Similarly, $e$ is an order unit for the ordered $*$-vector space $(V,V^+)$ if and only if $e$ is an order unit for the ordered real vector space $(V_h, V^+)$ as in Definition~\ref{ord-unit-real-def}; and $e$ is Archimedean for $(V,V^+)$ if and only if $e$ is Archimedean for $(V_h,V^+)$ as in Definition~\ref{Arch-ord-unit-real-def}.
\end{remark}

\subsection{Positive $\C$-linear functionals and states} \label{C-functionals-subsec}

We now turn our attention to positive maps, positive $\C$-linear functionals, and states on ordered $*$-vector spaces.  Many of the results here may be viewed as analogues of the results in \S\ref{subsec-positive-R-functionals}.  

\begin{definition}
Let $(V, V^+)$ be an ordered $*$-vector space with order unit $e$.  A $\C$-linear functional $f : V \to \C$ is \emph{positive} if $f (V^+) \subseteq [0,\infty)$ and a {\em state} if it is positive and $f(e)=1$.
\end{definition}

\begin{definition}
Let $(V, V^+)$ be an ordered $*$-vector space with order unit $e$, and let $(W,W^+)$ be an ordered $*$-vector space with order unit $e'$.  A linear map $\phi : V \to W$ is \emph{positive} if $v \in V^+$ implies $\phi(v) \in W^+$, and \emph{unital} if $\phi(e) = e'$.  A linear map $\phi : V \to W$ is an \emph{order isomorphism} if $\phi$ is bijective and $v \in V^+$ if and only if $\phi(v) \in W^+$.
\end{definition}

\begin{proposition}
Let $(V,V^+)$ be an ordered $*$-vector space with an order unit, and let $(W,W^+)$ be an ordered $*$-vector space.  If $\phi : V \to W$ is a positive linear map, then $\phi(v^*) = \phi(v)^*$ for all $v \in V$.
\end{proposition}

\begin{proof}
Since $\phi$ is positive and $V_h = V^+ - V^+$, we have that $\phi(V_h) \subseteq W_h$.  Thus for any $v \in V$, if we write $v = x+iy$ for $x, y \in V_h$, we see that $\phi(v^*) = \phi (x-iy) = \phi(x) - i \phi(y) = (\phi(x) + i \phi(y))^* = \phi(x + i y)^* = \phi(v)^*$.
\end{proof}

\begin{corollary} \label{pos-preserves-star}
If $(V,V^+)$ is an ordered $*$-vector space with an order unit and $f : V \to \C$ is a positive linear functional, then $f(v^*) = \overline{f(v)}$ for all $v \in V$.
\end{corollary}

\begin{definition} \label{complexification-of-f-def}
Let $(V, V^+)$ be an ordered $*$-vector space.  If $f: V_h \to \R$, then we define $\widetilde{f} : V \to \C$ by $$\widetilde{f} (v) := f (\operatorname{Re}(v)) + i f(\operatorname{Im}(v)).$$
\end{definition}

\begin{proposition} \label{complexification-linear-positive}
Let $(V, V^+)$ be an ordered $*$-vector space.  If $f : V_h \to \R$ is $\R$-linear, then $\widetilde{f} : V \to \C$ is $\C$-linear.  In addition, $f$ is positive if and only if $\widetilde{f}$ is positive, and $f$ is a state if and only if $\widetilde{f}$ is a state.
\end{proposition}

\begin{proof}
Let $\lambda \in \C$ and $v, w \in V$.  Write $\lambda = a+ib$ for $a,b \in \R$, and write $v = x+iy$ and $w = x' + i y'$ for $x,y,x',y' \in V_h$.  If $f$ is linear, then 
\begin{align*}
\widetilde{f}(\lambda v + w) &= \widetilde{f} ((a+ib)(x+iy) + (x'+iy')) \\
&= \widetilde{f} ((ax-by+x')+i(bx+ay+y')) \\
&= f(ax-by+x')+if(bx+ay+y') \\
&= (af(x) - bf(y) + f(x')) + i(bf(x)+af(y)+f(y')) \\
&= (a+ib) (f(x) + i f(y)) + (f(x') + i f(y')) \\
& = \lambda \widetilde{f} (v) + \widetilde{f} (w)
\end{align*}
so $\widetilde{f}$ is $\C$-linear.  In addition, since $\widetilde{f} (V_h) = f(V_h)$ we see that $f$ is positive if and only if $\widetilde{f}$ is positive, and since $f(e) = \widetilde{f} (e)$ we see that $f$ is a state if and only if $\widetilde{f}$ is a state.
\end{proof}

The following proposition allows us to characterize the positive linear functionals on ordered $*$-vector spaces with an Archimedean unit.

\begin{proposition} \label{pos-iff-restrict-pos}
Let $(V, V^+)$ be an ordered $*$-vector space with order unit $e$.  If $f : V \to \C$ is a $\C$-linear functional, then $f$ is positive if and only if $f = \widetilde{g}$ for some positive $\R$-linear functional $g : V_h \to \R$.  (See Definition~\ref{complexification-of-f-def}.)
\end{proposition}

\begin{proof}
Suppose $f = \widetilde{g}$ for some positive $\R$-linear functional $g : V_h \to \R$.  Then $f(V^+) = g(V^+) \subseteq [0,\infty)$ and $f$ is positive.

Conversely, suppose that $f : V \to \C$ is positive, so that $f(V^+) \subseteq [0,\infty)$.  Using the fact that $V^+$ is a full cone of $V_h$, for any $x \in V_h$ we may write $x = x^+ - x^-$ with $x^+, x^- \in V^+$.  Hence $f(x) = f(x^+) - f(x^-) \in \R$, and we see that $f(V_h) \subseteq \R$.  Let $g := f|_{V_h}$.  Then $g : V_h \to \R$ is a positive $\R$-linear functional.  Furthermore, for any $v \in V$ if we write $v = x + iy$ with $x, y \in V_h$, then we see that $$\widetilde{g} (v) = g(x) + i g(y) = f(x) + i f(y) = f (x +iy) = f (v)$$ so that $f = \widetilde{g}$.
\end{proof}

\begin{proposition} \label{states-sep-pts}
Let $(V, V^+)$ be an ordered $*$-vector space with Archimedean order unit $e$.  If $v \in V$ and $f (v) =0$ for every state $f : V \to \C$, then $v = 0$.
\end{proposition}

\begin{proof}
Write $v = x + iy$ with $x, y \in V_h$.  Then for any state $g : V_h \to \R$, we have that $\widetilde{g} : V \to \C$ is a state, and by hypothesis $\widetilde{g}(v) = 0$.   This implies $g(x) + i g(y) = 0$, which implies $g(x)=g(y)=0$.  But then Proposition~\ref{pos-R-funct-separating} implies that $x=y=0$, and hence $v=0$.
\end{proof}

\begin{proposition} \label{complex-states-pos-imply-pos}
Let $(V, V^+)$ be an ordered $*$-vector space with Archimedean order unit $e$.  If $v \in V$ and $f(v) \geq 0$ for every state $f : V \to \C$, then $v \in V^+$.
\end{proposition}

\begin{proof}
Write $v = x + iy$ with $x, y \in V_h$.  Then for any state $g : V_h \to \R$, we have that $\widetilde{g} : V \to \C$ is a state and by hypothesis, $\widetilde{g}(v) \geq 0$.   This implies $g(x) + i g(y) \geq 0$, which implies $g(x) \geq 0$ and $g(y)=0$.  But then Proposition~\ref{pos-R-funct-separating} implies that $y=0$, and Proposition~\ref{pos-for-states-implies-pos-prop} implies that $x \geq 0$.  Hence $v = x \geq 0$, and $v \in V^+$.
\end{proof}

\begin{remark}
Note that when $V$ is a $*$-vector space, we have $V = V_h \oplus i V_h$.  In fact, $V$ is the complexification of the real vector space $V_h$, and alternatively we may describe $V$ as the complex vector space $V_h \otimes_\R \C$.  Likewise, if $f : V_h \to \R$ is $\R$-linear, the $\C$-linear map $\widetilde{f} : V \to \C$ is the complexification of $f$.  The complexification is an (additive) functor from the category of real vector spaces to the category of complex vector spaces, and many of our results in this section may be viewed as extending results from the real ordered vector spaces to their complexifications.  We have chosen to give direct proofs rather than appeal to categorical arguments, since we feel these proofs are more straightforward and illuminate the relevant phenomena.
\end{remark}

\subsection{The Archimedeanization of an ordered $*$-vector space} \label{Arch-Complex-Vec-Spaces}

Here we extend the process of Archimedeanization for real vector spaces described in \S\ref{real-Arch-subsec} to complex $*$-vector spaces.  Let $(V,V^+)$ be an ordered $*$-vector space with order unit $e$.  Define $D := \{ v \in V_h : re + v \in V^+ \text{ for all $r > 0$} \}$, and $N_\R := D \cap -D$.  As in Remark~\ref{D-N-real-remark}, we have that $N_\R$ is a real subspace of $V_h$ and $$N_\R = \bigcap_{f : V_h \to \R \atop \text{is a state}} \ker f.$$  In analogy, let us define $$N :=  \bigcap_{f : V \to \C \atop \text{is a state}} \ker f.$$  Using Proposition~\ref{pos-iff-restrict-pos} one can easily verify that $N = N_\R \oplus i N_\R$.  One can also see from the definition that $N$ is a complex subspace of $V$ that is closed under the $*$-operation.  Thus we may form the quotient $V/N$ with the well-defined $*$-operation $(v+N)^* = v^* + N$.

Note that for any $v +N \in (V/N)_h$, we have that $v+N = v^*+N$, and hence $v+N = \frac{v+v^*}{2} + N = \operatorname{Re} (v) + N$.  Thus $(V/N)_h = \{ v + N : v \in V_h \}$.  We define the positive elements of $(V/N)_h$ to be the set $(V/N)^+ := \{ v+ N : v \in D \}$. One can easily verify that $((V/N)_h, (V/N)^+)$ is order isomorphic to $(V_h/N_\R, D+N_\R)$ via the map $v + N \mapsto v+N_\R$.  Thus, it follows from Theorem~\ref{V/N-is-Arch-prop} that $(V/N, (V/N)^+)$ is an ordered $*$-vector space with Archimedean order unit $e+N$.

\begin{definition}
Let $(V,V^+)$ be an ordered $*$-vector space with order unit $e$.  Define $D := \{ v \in V_h : re + v \in V^+ \text{ for all $r > 0$} \}$ and define $$N :=  \bigcap_{f : V \to \C \atop \text{is a state}} \ker f.$$  We define $V_\textnormal{Arch}$ to be the ordered $*$-vector space $(V/N, (V/N)^+)$ with the Archimedean order unit $e+N$.  We call $V_\textnormal{Arch}$ the \emph{Archimedeanization} of $V$.

\end{definition}

In addition, we have the following complex version of Theorem~\ref{real-Arch-charact-thm}.

\begin{theorem} \label{complex-Arch-charact-thm}
Let  $(V,V^+)$ be an ordered $*$-vector space with an order unit $e$, and let $V_\textnormal{Arch}$ be the Archimedeanization of $V$.  Then there exists a unital surjective positive linear map $q : V \to V_\textnormal{Arch}$ with the property that whenever $(W,W^+)$ is an ordered $*$-vector space with Archimedean order unit $e'$, and $\phi : V \to W$ is a unital positive linear map, then there exists a unique positive linear map $\widetilde{\phi} : V_\textnormal{Arch} \to W$ with $\phi = \widetilde{\phi} \circ q$.
\begin{equation*}
\xymatrix{  V_\textnormal{Arch} \ar@{-->}[rd]^{\widetilde{\phi}} & \\ V \ar[u]^q  \ar[r]^<>(.35)\phi & W}
\end{equation*}
In addition, this property characterizes $V_\textnormal{Arch}$:  If $V'$ is any ordered $*$-vector space with an Archimedean order unit and $q' : V \to V'$ is a unital surjective positive linear map with the above property, then $V'$ is isomorphic to $V_\textnormal{Arch}$ via a unital order isomorphism $\psi : V_\textnormal{Arch} \to V'$ with $\psi \circ q = q'$.
\end{theorem}

\begin{proof}
Write $V_\textnormal{Arch} = V/N$, and let $q : V \to V/N$ be the quotient map $q(v) = v + N$.  If $(W,W^+)$ is an ordered $*$-vector space with Archimedean order unit $e'$, and $\phi : V \to W$ is a unital positive linear map, then for any $v \in N$ we may write $v = \operatorname{Re} (v) + i \operatorname{Im} (v)$ with $\operatorname{Re} (v), \operatorname{Im} (v) \in N_\R = D \cap -D$.  Thus for $\operatorname{Re} (v)$ we have that $re + \operatorname{Re} (v) \in V^+$ and $re- \operatorname{Re} (v) \in V^+$ for all $r >0$.  Applying $\phi$ we then have that $re' + \phi(\operatorname{Re} (v)) \in W^+$ and $re'- \phi(\operatorname{Re} (v)) \in W^+$ for all $r >0$.  Since $e'$ is an Archimedean order unit, it follows that $\phi(\operatorname{Re} (v)) = 0$.  A similar argument shows that $\phi(\operatorname{Im} (v)) = 0$.  Thus $\phi(v) = \phi(\operatorname{Re} (v)) + i \phi(\operatorname{Im} (v)) = 0$, and $\phi$ vanishes on $N$.  It follows that $\widetilde{\phi} : V/N \to W$ given by $\widetilde{\phi} (v+N) = \phi(v)$ is well defined.  The rest of the proof is the same as the proof of Theorem~\ref{real-Arch-charact-thm}.
\end{proof}

\subsection{Quotients of ordered $*$-vector spaces} \label{Complex-quotients-subsec}

In this section we describe a process for forming quotients of ordered $*$-vector spaces and Archimedean $*$-vector spaces.  Many of these results may be viewed as complex versions of the results in \S\ref{Real-quotients-subsec}.

\begin{definition} 
Let $(V,V^+)$ be an ordered $*$-vector space.  A complex subspace $J \subseteq V$ is called \emph{self adjoint} if $J=J^*$.  A self-adjoint subspace $J$ is called an \emph{order ideal} if $p \in J \cap V^+$ and $0 \leq q \leq p$ implies that $q \in J$.
\end{definition}

\begin{remark}
Note that if $J$ is a self-adjoint subspace of an ordered $*$-vector space $V$, then we may define a $*$-operation on $V/J$ by $(v + J)^* = v^* + J$.  Furthermore, if $v + J \in (V/J)_h$, then $v + J = v^* + J$ and $v + J = \frac{v + v^*}{2} + J$.  This shows that $(V/J)_h = V_h + J$, where $V_h + J = \{ v + J : v \in V_h \}$.  It also follows that $V^+ + J \subseteq (V/J)_h$.  In addition, if we define $J_\R := J \cap V_h$, then $J_\R$ is a real subspace of $V$ and $J = J_\R \oplus i J_\R$.  
\end{remark}

The proof of the following proposition is identical to that of Proposition~\ref{order-quotients-prop}.

\begin{proposition} Let $(V,V^+)$ be an ordered $*$-vector space with order unit $e$ and let $J \subseteq V$ be an order ideal. Then $(V/J, V^++J)$ is an ordered $*$-vector space with order unit $e+J.$
\end{proposition}

\begin{definition} Let $(V,V^+)$ be an ordered $*$-vector space with order unit $e$, and let $J \subseteq V$ be an order ideal. We call the ordered vector space $(V/J, V^++J)$ with order unit $e+J$, the \emph{quotient of $V$ by $J$}.   If, in addition, $e$ is an Archimedean order unit, then we call the Archimedeanization of $(V/N, V^++J)$ the \emph{Archimedean quotient of $V$ by $J$}.
\end{definition}

In analogy with Proposition~\ref{Arch-one-quotient-prop}, we would like to view the Archimedean quotient of a complex $*$-vector space by an order ideal as a single quotient.

\begin{proposition} \label{complex-Arch-one-quotient-prop}
Let $(V,V^+)$ be an ordered $*$-vector space with order unit $e$, and let $J$ be an order ideal of $V$.  If we define 
\begin{align*}
N_{J_\R} := \{ v \in V_h : \text{for all } &r > 0 \text{ there exists } j,k \in J_\R  \text{ such that } \\
& j + re + v \in V^+ \text{ and } k + re - v \in V^+ \},
\end{align*}
and set $N_J := N_{J_\R} \oplus i N_{J_\R}$, then $N_J$ is an order ideal of $V$.  Furthermore, if we let 
\begin{align*}
(V/N_J)^+ := \{ v + N_J : \text{for all } &r > 0 \text{ there exists } j \in J \\
& \text{ such that } j + re + v \in V^+  \}
\end{align*}
then $(V/N_J, (V/N_J)^+ )$ is an ordered vector space with Archimedean order unit $e + N_J$, and the Archimedean quotient of $V$ by $J$ is unitally order isomorphic to $(V/N_J, (V/N_J)^+ )$.
\end{proposition}

\begin{proof}
It follows from Proposition~\ref{Arch-one-quotient-prop} that $N_{J_\R}$ is a real order ideal in $V_h$, and thus $N_J := N_{J_\R} \oplus i N_{J_\R}$ is an order ideal of $V$.

Furthermore, we see that $((V/N_J)_h, (V/N_J)^+ )$ is unitally order isomorphic to $(V_h / N_{J_\R}, (V_h / N_{J_\R})^+)$ via the map $v + N_J \mapsto v + N_{J_\R}$.  Since Proposition~\ref{Arch-one-quotient-prop} shows that $(V_h / N_{J_\R}, (V_h / N_{J_\R})^+)$ is an ordered vector space with Archimedean order unit $e + N_{J_\R}$, it follows that $(V/N_J, (V/N_J)^+ )$ is an ordered $*$-vector space with Archimedean order unit $e + N_J$.

Finally, let $D := \{ v + J \in (V/J)_h : (re + v)+ J \in V^+ + J \text{ for all $r >0$}\}$, let $N_\R := D \cap -D$, and let $N := N_\R \oplus i N_\R$.  Then, by definition, the Archimedean quotient of $V$ by $J$ is the ordered vector space $( (V/J)/N, D + N)$ with Archimedean order unit $(e+J) + N$.  If we define $\psi : (V/J)/N \to V / N_J$ by $\psi ((v+J) +N) = v + N_J$, then an argument as in the proof of Proposition~\ref{Arch-one-quotient-prop} shows that $\psi$ is an order isomorphism.
\end{proof}

We also have the following complex version of the ``First Isomorphism Theorem" for positive linear maps.  The proof is similar to that of Theorem~\ref{first-iso-real-thm} and therefore we omit it.

\begin{theorem} \label{first-iso-complex-thm}
Let $(V,V^+)$ be an ordered $*$-vector space with Archimedean order unit $e$, and let $(W,W^+)$ be an ordered $*$-vector space with Archimedean order unit $e'$.  If $\phi : V \to W$ is a unital positive linear map, then $\ker \phi$ is an order ideal and the Archimedean quotient of $V$ by $\ker \phi$ is unitally order isomorphic to $V / \ker \phi$ with positive cone 
$$(V / \ker \phi)^+ = \{ v + \ker \phi : \text{$\forall$ $r >0$ $\exists$ $j \in \ker \phi$ such that $j + re + v \in V^+$ } \} $$ and Archimedean order unit $e + \ker \phi$.  In addition, the map $\widetilde{\phi} : V / \ker \phi \to W$ given by $\widetilde{\phi} (v + \ker \phi) = \phi (v)$ is a unital positive linear map.  Moreover, if $\phi (V^+) = W^+ \cap \im \phi$, then $\widetilde{\phi}$ is an order isomorphism from $V / \ker \phi$ onto $\im \phi$.
\end{theorem}

\section{Seminorms on ordered $*$-vector spaces} \label{metric-from-order-unit}

In \S \ref{seminorm-from-order-unit} we defined the order seminorm on an ordered real vector space containing an order unit, and in Theorem~\ref{norm-char} we characterized the order seminorm as the unique seminorm satisfying certain conditions.  If $(V,V^+)$ is an ordered $*$-vector space with order unit $e$, the real subspace $V_h$ is endowed with the order seminorm.  In this section, we are interested in extending the order seminorm on $V_h$ to a seminorm on $V$ that preserves the $*$-operation.  We will show that, in general, there can be many such seminorms on $V$ that do this, and that among all such seminorms there is a minimal one and a maximal one.

\begin{definition}
If $V$ is a $*$-vector space, a seminorm (respectively, norm) $\| \cdot \|$ on $V$ is called a \emph{$*$-seminorm} (respectively, \emph{$*$-norm}) if $\| v^* \| = \| v \|$ for all $v \in V$.
\end{definition}

\begin{lemma} \label{real-im-smaller}
Let $V$ be a $*$-vector space and let $\| \cdot \|$ be a $*$-seminorm on $V$.  Then for any $v \in V$ we have that $\| \operatorname{Re} (v) \| \leq \| v \|$ and $\| \operatorname{Im}(v) \| \leq \| v \|$.
\end{lemma}

\begin{proof}
We see that $\| \operatorname{Re}(v) \| = \frac{1}{2} \| \operatorname{Re}(v) + i \operatorname{Im}(v) +  \operatorname{Re}(v) - i \operatorname{Im}(v) \| \leq \frac{1}{2} \| \operatorname{Re}(v) + i \operatorname{Im}(v) \| + \frac{1}{2} \| \operatorname{Re}(v) - i \operatorname{Im}(v) \| = \frac{1}{2} \| v \| + \frac{1}{2} \| v^* \| = \| v \|$.  Similarly, $\| \operatorname{Im}(v) \| = \frac{1}{2} \| \operatorname{Re}(v) + i \operatorname{Im}(v) -  \operatorname{Re}(v) + i \operatorname{Im}(v) \| \leq \frac{1}{2} \| \operatorname{Re}(v) + i \operatorname{Im}(v) \| + \frac{1}{2} \| \operatorname{Re}(v) - i \operatorname{Im}(v) \| = \frac{1}{2} \| v \| + \frac{1}{2} \| v^* \| = \| v \|$. 
\end{proof}

\begin{definition} \label{order-norms-def}
 Let $(V,V^+)$ be an ordered $*$-vector space with order unit $e$, and let $\| \cdot \|$ be the order seminorm on $V_h$.  An \emph{order seminorm on $V$} is a $*$-seminorm $||| \cdot |||$ on $V$ with the property that $||| v ||| = \| v \|$ for all $v \in V_h$.
\end{definition}

In general, there are many order seminorms on an ordered $*$-vector space.  In this section we will examine three particular order seminorms (viz., the minimal, the maximal, and the decomposition seminorms) that play an important role in the analysis of ordered $*$-vector spaces.

\subsection{The minimal order seminorm $\| \cdot \|_m$} \label{minimal-seminorm-subsec}

\begin{definition} \label{min-defn}
Let $(V, V^+)$ be an ordered $*$-vector space with order
unit $e$.  We define the \emph{minimal order seminorm} $\| \cdot \|_m : V \to [0,\infty)$ by $$\|v\|_m := \sup \{ |f(v)|: f:V \to \C \text{ is a state} \}.$$
\end{definition}

The following result justifies the terminology in the above definition.

\begin{theorem} \label{min-norm}  Let $(V,V^+)$ be an ordered $*$-vector space with order unit $e$, and let $\| \cdot \|$ denote the order seminorm on $V_h$. The following are true:
\begin{enumerate}
\item $\| \cdot \|_m$ is a $*$-seminorm on the complex $*$-vector space $V$,
\item $\|v\|_m=\|v\|$ for all $v \in V_h$, and
\item if $|||\cdot |||$ is any other $*$-seminorm on the complex vector space $V$ such that $|||v||| = \|v\|$ for every $v \in V_h,$ then $\|v\|_m \leq |||v|||$ for every $v \in V$.
\end{enumerate}
\end{theorem}
\begin{proof}
We first note that for any $\lambda \in \C$ and $v \in V$, we have that $\| \lambda v \|_m = \sup \{ |f(\lambda v)|: f \text{ is a state} \} = \sup \{ |\lambda||f(v)|: f \text{ is a state} \} = |\lambda| \ \|v\|_m.$  In addition, given $v,w \in V,$ we have that 
\begin{align*}
\|v+w\|_m &= \sup \{|f(v+w)|: f \text{ is a state} \} \\
&\le \sup \{|f(v)| + |f(w)|: f \text{ is a state} \} \le \|v\|_m + \|w\|_m.
\end{align*}
Thus $\| \cdot \|_m$ is a complex seminorm.  In addition, Corollary~\ref{pos-preserves-star} shows that for any $v \in V$ we have $\| v^* \|_m = \sup \{ |f(v^*)|: f:V \to \C \text{ is a state} \} =  \sup \{ |\overline{f(v)}|: f:V \to \C \text{ is a state} \} =  \sup \{ |f(v)|: f:V \to \C \text{ is a state} \} = \| v \|_m$.  Hence $\| \cdot \|_m$ is a $*$-seminorm.

If $v \in V_h,$ then by Proposition~\ref{seminorm} and Proposition~\ref{pos-iff-restrict-pos} we see that
\begin{align*}
\|v\|_m &= \sup \{ |f(v)|: f:V \to \C \text{ is a state} \} \\
&= \sup \{ |\widetilde{g}(v)|:  g:V_h \to \R \text{ is a state} \} \\
&= \sup \{ |g(v)|:  g:V_h \to \R \text{ is a state} \} \\
&= \|v\|.
\end{align*}

Finally, assume that $||| \cdot |||$ is a $*$-seminorm satisfying $|||v|||= \|v\|$ for every $v \in V_h.$  Let $v \in V$ and let $f:V \to \C$ be a state.  Choose $\theta \in \R,$ such that $|f(v)| = e^{i \theta} f(v) = f(e^{i \theta}v).$  Let $w=e^{i \theta}v$.  Since $f$ is positive we have $f(\operatorname{Re}(w)), f(\operatorname{Im}(w)) \in \R$.  Also, because $f(w) = | f(v)| \in [0,\infty)$, we have that $f(w) = f(\operatorname{Re}(w)) + i f(\operatorname{Im}(w))$, and it follows that $f(\operatorname{Im}(w)) = 0$ and $f(\operatorname{Re}(w)) = f(w)$.  Hence $|f(v)|= f(w) = f(\operatorname{Re}(w)) \leq \|\operatorname{Re}(w)\| = |||\frac{w+ w^*}{2}||| \leq (1/2) ( ||| w ||| + ||| w^* |||) \leq |||w|||= |||v|||$. Taking the supremum over all states yields $\|v\|_m \le |||v|||$.
\end{proof}

\subsection{The maximal order seminorm $\| \cdot \|_M$} \label{maximal-seminorm-subsec}

\begin{definition} Let $(V,V^+)$ be an ordered $*$-vector space with order unit $e$, and let $\| \cdot \|$ denote the order seminorm on $V_h$.  We define the \emph{maximal order seminorm} $\| \cdot \|_M:V \to [0,\infty)$ by
$$\|v\|_M = \inf \left\{ \sum_{i=1}^n |\lambda_i| \, \|v_i\| \ : \ v = \sum_{i=1}^n \lambda_i v_i \text{ with } v_i \in V_h \text{ and } \lambda_i \in \C \right\}.$$
(Note that the above set is nonempty since we have $v = \operatorname{Re} (v) + i \operatorname{Im}(v)$ for any $v \in V$.)
\end{definition}

The following result justifies the terminology in the above definition.

\begin{theorem} \label{max-norm}  Let $(V,V^+)$ be an ordered $*$-vector space with order unit $e$. The following are true:
\begin{enumerate}
\item $\| \cdot \|_M$ is a $*$-seminorm on the complex $*$-vector space $V$,
\item $\|v\|_M=\|v\|$ for all $v \in V_h$, and 
\item if $|||\cdot |||$ is any other $*$-seminorm on the complex vector space $V$ such that $|||v||| = \|v\|_M$ for every $v \in V_h$, then $|||v||| \le \|v\|_M$ for every $v \in V$.
\end{enumerate}
\end{theorem}
\begin{proof} We first show that $\| \cdot \|_M$ is a $*$-seminorm on $V$.  To begin, the fact that $\| \lambda v\|_M = |\lambda| \|v\|_M$ follows directly from the definition and the fact that $v = \sum_i \lambda_i v_ i$ if and only if $\lambda v =  \sum_i (\lambda \lambda_i) v_ i$.  In addition, if $v, w \in V$ and if $v= \sum_i \lambda_i v_i$ and $w= \sum_j \mu_j w_j,$ where both sums are finite with $\lambda_i, \mu_j \in \C$ and $v_i, w_j \in V_h$ for every $i,j$, then $v+w = \sum_i \lambda_i v_i + \sum_j \mu_j w_j$ is among the ways (but not necessarily every way) to express $v+w$ as an appropriate type of sum.  From this it follows that
\begin{align*} 
&\|v+w\|_M \\
\leq &\inf \left\{ \sum_i |\lambda_i| \|v_i\| + \sum_j |\mu_j|\|w_j\| :
v= \sum_i \lambda_i v_i \text{ and } w= \sum_j \mu_j w_j \right\} \\
= &\|v\|_M + \|w\|_M.
\end{align*}
Thus $\|\cdot \|_M$ is a complex seminorm on $V$.  Furthermore, if $v \in V$, then $v = \sum_i \lambda_i v_i$ with each $v_i \in V_h$ if and only if $v^* = \sum_i \overline{\lambda}_i v_i$.  Because $| \overline{\lambda_i}| = | \lambda_i |$, it follows from the definition of the maximal order norm that $\| v \|_M = \| v^* \|_M$.  Hence $\| \cdot \|_M$ is a $*$-seminorm.

To see the second property, let $v \in V_h$ and write $v= \sum_i \lambda_i v_i$ with $\lambda_i \in \C$ and $v_i \in V_h$ for every $i.$ We have that $v= v^* = \sum_i \overline{\lambda_i} v_i,$ and hence,
$v= \sum_i \operatorname{Re}(\lambda_i) v_i.$  Thus, using that $\| \cdot \|$ is a real seminorm, we have $\|v\| \leq \sum_i |\operatorname{Re}(\lambda_i)| \|v_i\| \le \sum_i |\lambda_i| \|v_i\|,$ and taking the infimum over all such sums yields $\|v\| \le \|v\|_M.$ The inequality $\|v\|_M \le \|v\|$ follows by taking the trivial sum $v=1v.$  Hence $\| v \|_M = \|v \|$ for all $v \in V_h$.

Finally, assume that $||| \cdot |||$ is a $*$-seminorm satisfying $|||v|||= \|v\|$ for every $v \in V_h$.  Let $v \in V$.  Given any expression of $v= \sum_i \lambda_i v_i$ with $\lambda_i \in \C$ and $v_i \in V_h$ for all $i$, we have that $|||v||| \leq \sum_i |\lambda_i| \, |||v_i||| = \sum_i |\lambda_i| \, \|v_i\|$, which after taking the infimum of the right-hand side yields $|||v||| \le \|v\|_M.$
\end{proof}

\begin{remark} \label{min-max-order-norms}
Theorem~\ref{min-norm} and Theorem~\ref{max-norm} show that $\| \cdot \|_m$ and $\| \cdot \|_M$ are order seminorms on $V$, and if $||| \cdot |||$ is any order seminorm on $V$, then $\| v \|_m \leq ||| v ||| \leq \| v \|_M$ for all $v \in V$.
\end{remark}

\begin{proposition} \label{all-order-norms-equiv}
Let $(V,V^+)$ be an ordered $*$-vector space with order unit $e$.  Then any two order seminorms on $V$ are equivalent. Moreover, if $||| \cdot |||$ is any order seminorm on $V$, then $\{v \in V: |||v|||=0 \} = \bigcap_{ f : V \to \C \atop \text{is a state}} \ker f$.
\end{proposition}

\begin{proof}
In light of Remark~\ref{min-max-order-norms}, it suffices to show that the maximal order seminorm $\| \cdot \|_M$ and the minimal order seminorm $\| \cdot \|_m$ are equivalent.  Let $v \in V$.  By Theorem~\ref{min-norm} and Theorem~\ref{max-norm} we have that $\| v \|_m \leq \| v \|_M$.  In addition, if we write $v = \operatorname{Re}(v) + i \operatorname{Im}(v)$ and let $\| \cdot \|$ denote the order seminorm on $V_h$, then using Lemma~\ref{real-im-smaller} we have $\| v \|_M  \leq \| \operatorname{Re} (v) \| + | i | \| \operatorname{Im} (v) \| =  \| \operatorname{Re} (v) \|_m +  \| \operatorname{Im} (v) \|_m \leq \| v \|_m + \| v \|_m = 2 \| v \|_m$.  Hence $\| \cdot \|_m$ and $\| \cdot \|_M$ are equivalent seminorms.

To see the last statement, we note that since $||| \cdot |||$ and $\| \cdot \|_m$ are equivalent we have $$\{ v \in V: |||v|||=0 \} = \{ v \in V: \|v\|_m =0 \} = \bigcap_{ f : V \to \C \atop \text{is a state}} \ker f .$$ 
\end{proof}

\begin{remark}
Proposition~\ref{all-order-norms-equiv} shows that in the Archimedeanization of an ordered $*$-vector space (see \S\ref{Arch-Complex-Vec-Spaces}) the subspace $N := \bigcap_{ f : V \to \C \atop \text{is a state}} \ker f$ is equal to $\{ v \in V: |||v|||=0 \}$ for any order seminorm $||| \cdot |||$ on $V$.
\end{remark}

\begin{proposition}
Let $(V,V^+)$ be an ordered $*$-vector space with order unit $e$.  Then the following are equivalent:
\begin{enumerate}
\item The order seminorm on $V_h$ is a norm.
\item $\displaystyle \bigcap_{ f : V_h \to \R \atop \text{is a state}} \ker f = \{ 0 \}$.
\item $\displaystyle \bigcap_{ f : V \to \C \atop \text{is a state}} \ker f = \{ 0 \}$.
\item Some order seminorm on $V$ is a norm.
\item All order seminorms on $V$ are norms.
\end{enumerate}
\end{proposition}

\begin{proof}
$(1) \Longrightarrow (2)$.  Suppose the order seminorm $\| \cdot \|$ on $V_h$ is a norm.  If $v \in \displaystyle \bigcap_{ f : V_h \to \R \atop \text{is a state}} \ker f$, then it follows from Proposition~\ref{seminorm} that $\| v \| = 0$ and hence $v = 0$.

\noindent $(2) \Longrightarrow (3)$.  Suppose $\displaystyle \bigcap_{ f : V_h \to \R \atop \text{is a state}} \ker f = \{ 0 \}$.  Let $v \in \displaystyle \bigcap_{ f : V \to \C \atop \text{is a state}} \ker f$.  Then for any state $f : V_h \to \R$, it follows from Proposition~\ref{complexification-linear-positive} that $\widetilde{f} : V \to \C$ is a state.  Hence $f(\operatorname{Re}(v)) + i f(\operatorname{Im}(v)) = \widetilde{f} (v) = 0$, and $f(\operatorname{Re}(v)) = f(\operatorname{Im}(v)) = 0$.  Since $f$ was arbitrary, we have that $\operatorname{Re}(v), \operatorname{Im}(v) \in \displaystyle \bigcap_{ f : V_h \to \R \atop \text{is a state}} \ker f = \{ 0 \}$.  Thus $v = 0$.

\noindent $(3) \Longrightarrow (4)$.  It follows from Definition~\ref{min-defn} that the minimal order seminorm $\| \cdot \|_m$ is a norm.

\noindent $(4) \Longrightarrow (5)$.  This follows from Proposition~\ref{all-order-norms-equiv}.

\noindent $(5) \Longrightarrow (1)$. This follows from the fact that the order seminorm on $V_h$ is the restriction of any order seminorm on $V$.

\end{proof}

\begin{corollary}
Let $(V,V^+)$ be an ordered $*$-vector space with an Archimedean order unit $e$.  Then every order seminorm on $V$ is a norm.
\end{corollary}

\begin{definition}
In the case when one (and hence all) of the order seminorms on $V$ is a norm, we will refer to it as an \emph{order norm}.   In this case, we also call $\| \cdot \|_m$ the \emph{minimal order norm on $V$}, we call $\| \cdot \|_M$ the \emph{maximal order norm on $V$}, and we call $\| \cdot \|_\dec$ the \emph{decomposition norm on $V$}.
\end{definition}

\begin{definition} \label{complex-order-top-def}
Let $(V,V^+)$ be an ordered $*$-vector space with an order unit $e$.   The \emph{order topology} on $V$ is the topology generated by any order seminorm on $V$.  Proposition~\ref{all-order-norms-equiv} shows that this topology is independent of which order seminorm is used.
\end{definition}

\begin{remark}
Since any order norm on $V$ is a $*$-norm, we see that $V_h$ is a closed subset of $V$.  In addition, because any order norm on $V$ restricts to the order norm on $V_h$, the subspace topology on $V_h$ is equal to the topology induced by the order norm on $V_h$.
\end{remark}

\begin{lemma} \label{pos-fct-norm-f(e)-lem}
Let $(V,V^+)$ be an ordered $*$-vector space with an Archimedean order unit $e$.  Let $||| \cdot |||$ be any order norm on $V$ and $f : V \to \C$ be a positive functional.  If $\| f \|$ denotes the norm of the functional $f$ with respect to the order norm $||| \cdot |||$, then $\| f \| = |f(e)|$.
\end{lemma}

\begin{proof}
Suppose that $f(e) = 0$.  Then by Proposition~\ref{positive-R-functionals-cts} $f|_{V_h}=0$ and by Proposition~\ref{pos-iff-restrict-pos} we have $f = 0$, and the claim holds.  If $f(e) \neq 0$, then $g := (1/f(e)) f$ is a state, and for any $v \in V$ we have that 
\begin{align*}
| f(v) | &= f(e) | g(v) | \leq f(e) \sup \{ | h(v) | : h : V \to \C  \text{ is a state} \} \\
&\leq f(e) \| v \|_m \leq f(e) ||| v |||
\end{align*}
so that $\| f \| \leq f(e)$.  In addition, since $||| e ||| = \| e \| = 1$, it follows that $\| f \| = f(e)$.
\end{proof}

\begin{corollary} \label{pos-fct-order-cts-cor}
Let $(V,V^+)$ be an ordered $*$-vector space with an Archimedean order unit $e$.  If $f : V \to \C$ is a positive functional, then $f$ is continuous with respect to the order topology on $V$.
\end{corollary}

\subsection{The decomposition seminorm $\| \cdot \|_\dec$} \label{decomp-seminorm-subsec}

$ $

\smallskip

In addition to the minimal and maximal order seminorms, there is another order seminorm that will play a central role in our study of ordered $*$-vector spaces.

\begin{definition} \label{decomp-seminorm} Let $(V,V^+)$ be an ordered $*$-vector space with order unit $e$, and let $\| \cdot \|$ denote the order norm on $V_h$.  We define the \emph{decomposition seminorm} $\| \cdot \|_\dec :V \to [0,\infty)$ by $$ \|v\|_{\dec} := \inf \left\{ \left\| \sum_{i=1}^n |\lambda_i|p_i \right\|: v = \sum_{i=1}^n \lambda_i p_i \text{ with } p_i \in V^+ \text{ and } \lambda_i \in \mathbb{C} \right\}.$$
Note that the above set is nonempty since for any $v \in V$ we may write $\operatorname{Re} (v) = p_1 - p_2$ for $p_1, p_2 \in V^+$ and we may write $\operatorname{Im} (v) = q_1 - q_2$ for $q_1, q_2 \in V^+$, so that $v = p_1 - p_2 + i q_1 - i q_2$.
\end{definition}

\begin{proposition}
Let $(V,V^+)$ be an ordered $*$-vector space with order unit $e$. Then the decomposition seminorm $\| \cdot \|_{\dec}$ is an order seminorm on $V$.
\end{proposition}
\begin{proof} First, it is straightforward to prove that the decomposition seminorm is a seminorm. Also, if $v = \sum_i \lambda_i p_i$ is a sum as above, then $v^* = \sum \overline{\lambda_i} p_i,$ and from this it follows that $\|v^*\|_{\dec} =\|v\|_{\dec},$ so the decomposition seminorm is a $*$-seminorm.

Thus it remains to show that if $h \in V_h,$ then $\|h\|=\|h\|_{\dec}.$ To this end, we fix $\epsilon > 0$ and choose $r \in \mathbb{R}$ with $\|h\|+ \epsilon > r > \|h\|$.  Then $h = \frac{1}{2}( re+h) + \frac{-1}{2}(re-h)$ is an expression of $h$ as a linear combination of positives, and hence $\|h\|_{\dec} \le \| |\frac{1}{2}|(re+h) + |\frac{-1}{2}|(re-h) \| = r.$ Since $\epsilon > 0$ was arbitrary, we have that $\|h \|_{\dec} \le \|h\|.$

Conversely, if $h = \sum_i \lambda_i p_i$ is any expression of the above form, then
$h=h^* = \sum_i \overline{\lambda_i} p_i$ and hence, $h = \sum_i \operatorname{Re}(\lambda_i) p_i \le \sum_i |\lambda_i|p_i.$  Similarly, $-h \le \sum_i |\lambda_i| p_i$ and so Theorem~\ref{norm-char}(2) implies that $\|h\| \le \|\sum_i |\lambda_i| p_i \|$.
Since this inequality holds for all such sums, we have that $\|h\| \leq \|h\|_{\dec}$.  Thus $\|h\| = \|h\|_{\dec}$
\end{proof}

We now turn our attention to positive linear maps and their relationship with the order seminorms. 

\begin{lemma} \label{technical-min-dec-lem}
Let $(V,V^+)$ be an ordered $*$-vector space with order unit $e$.  If $\lambda \in \C$ and $h \in V_h$, then $\| \lambda e + h \|_m = \| \lambda e + h \|_\dec$.
\end{lemma}

\begin{proof}
Let $\alpha := \inf \{ f(h) : f \text{ is a state on } V \}$ and let $\beta := \sup \{ f(h) : f \text{ is a state on } V \}$.  Also set $r := \frac{ \alpha + \beta}{2}$, $\mu := \lambda + r$, and $k := h - re$.  Then $\lambda e + h = \mu e + k$, and it suffices to show that $\| \mu e + k \|_m = \| \mu e + k \|_\dec$.  Note, by the choice of $r$, that if we let $b := \frac{\beta - \alpha}{2}$, then $$ \sup \{ f(k) : f \text{ is a state on } V \} = \beta - r = \frac{\beta - \alpha}{2} = b \geq 0$$ and $$ \inf \{ f(k) : f \text{ is a state on } V \} = \alpha - r = \frac{\alpha - \beta}{2} = -b \leq 0.$$  Also note that
\begin{align*}
\| \mu e + k \|_m &= \sup \{ | f (\mu e+k) | : f \text{ is a state on } V \} \\
&=  \{ | \mu +f(k) | : f \text{ is a state on } V \} \\
&= \max \{ | \mu + b |, |\mu - b| \} \qquad \text{(since $-b \leq f(k) \leq b$ for all states $f$).}
\end{align*} 
Because $\| \mu e + k \|_m \leq \| \mu e + k \|_\dec$, we need only show that $\| \mu e + k \|_\dec \leq \max  \{ \mu + b, \mu - b \}$.  To this end let $c > b$ so that $ce \pm k \geq 0$ and 
$$\mu e + k = \frac { \mu + c }{2c} (ce + k) + \frac{\mu-c}{2c} (ce-k),$$ which implies that
\begin{align*}
& \| \mu e + k \|_\dec \\
\leq & \left\| \, \left| \frac { \mu + c }{2c} \right| (ce + k) + \left| \frac{\mu-c}{2c} \right| (ce-k) \right\| \\
= & \left\| \, \left( \frac {| \mu + c |}{2}+ \frac{|\mu - c |}{2} \right) e + \left( \left| \frac{\mu + c}{2c}\right| - \left| \frac{\mu-c}{2c} \right| \right) k \right\| \\
= & \sup \left\{  \left| \, \left( \frac {| \mu + c |}{2}+ \frac{|\mu - c |}{2} \right) + \left( \left| \frac{\mu + c}{2c}\right| - \left| \frac{\mu-c}{2c} \right| \right) f(k) \right| \ : f \text{ is a state} \right\} \\
= & \max \left\{  \left| \, \left( \frac {| \mu + c |}{2}+ \frac{|\mu - c |}{2} \right) \pm  \frac{b}{c} \left( \frac{| \mu + c |}{2} - \frac{|\mu-c|}{2} \right) \, \right| \right\}.
\end{align*}
Since the above inequality holds for all $c > b$, we may conclude that it also holds for $c = b$ and thus we have 
\begin{align*}
\| \mu e + k \|_\dec &\leq \max \left\{ \left( \frac {| \mu + b |}{2}+ \frac{|\mu - b |}{2} \right) \pm  \left( \frac{|\mu + b|}{2} - \frac{|\mu-b|}{2} \right)  \right\} \\
&= \max \{ | \mu + b |, | \mu - b | \} \\
&= \| \mu e + k \|_m. 
\end{align*}
\end{proof}

\begin{definition}
Given a linear map $\phi: V \to W$ between ordered $*$-vector spaces, we let $\|\phi\|_m$ (respectively, $\| \phi \|_\dec$) denote the seminorm of the map $\phi$ when both of $V$ and $W$ are given the minimal seminorm (respectively, the decomposition seminorm).
\end{definition}

\begin{theorem} Let $(V,V^+)$ be an ordered $*$-vector space with Archimedean order unit $e$, and let $(W,W^+)$ be an ordered $*$-vector space with Archimedean order unit $e'$.  If $\phi:V \to W$ is a unital linear map, then the following are equivalent: 
\begin{enumerate}
\item $\phi$ is positive,
\item $\|\phi\|_m =1$,
\item $\|\phi\|_{\dec} =1$.
\end{enumerate}
\end{theorem}

\begin{proof}
\noindent $(1) \Longrightarrow (2)$.  Suppose that $\phi$ is positive.  If $f:W \to \mathbb{C}$ is a state, then $f \circ \phi:V \to \mathbb{C}$ is a state, and hence $\|\phi(v)\|_m = \sup \{ |f(\phi(v))| : f \text{ is a state on $W$} \} \leq \sup \{ |g(v)| : g \text{ is a state on $V$} \} = \|v\|_m,$ so that $\|\phi\|_m \leq 1$.  Since $\phi$ is unital, $\|\phi\|_m =1.$

\noindent $(2) \Longrightarrow (1)$.  Suppose that $\|\phi\|_m = 1$, and let $p \in V^+$. If $\phi(p) \notin W^+$, then there exists a state $f$ on $W$ such that, $\alpha = f(\phi(p)) \notin \mathbb{R}^+$.  If we consider the convex set $S := \{ t \in \C : 0 \leq t \leq \| p \| \}$ in $\C$, then since any convex set is the intersection of all circles containing it, there exists a circle containing $S$ but not containing $\alpha$.  Let $c \in \C$ be the center of this circle and let $r$ be its radius.  Note that $c$ and $r$ have the properties that $|c - \alpha| > r$, and $0 \le t \le \|p\|$ implies $|c-t| \leq r$. For any state $g : V \to \C$ we have that $0 \leq g(p) \leq \| p \|$, and thus $| g(ce - p ) | = | c - g(p) | \leq r$.  Taking the supremum over all states $g$ yields $\| ce - p \|_m < r$.  In addition, since $f$ is a state on $W$ we obtain $\|\phi(ce -p)\|_m \ge |f(\phi((ce -p))| = |c - \alpha|>r$.  But the fact that $\| ce - p \|_m \leq r$ and $\|\phi(ce -p)\|_m > r$ contradicts that $\| \phi \|_m = 1$.  Thus we must have that $\phi(p) \in W^+$, and $\phi$ is positive.

\noindent $(1) \Longrightarrow (3)$.   Suppose that $\phi$ is positive.  Let $v \in V$ and write $v = \sum_i \lambda_i p_i$ with $p_i \in V^+$. Since $\phi$ is positive linear we have $\phi(v) = \sum_i \lambda_i \phi(p_i)$ with $\phi(p_i) \in W^+$.  Thus $\|\phi(v)\|_{\dec} \leq \|\sum_i | \lambda_i | \phi(p_i) \| = \| \phi(\sum_i |\lambda_i| p_i) \| \leq \| \sum_i |\lambda_i| p_i \|$.  Taking the infimum over the right hand side, yields $\|\phi(v)\|_{\dec} \leq \|v\|_{\dec}$. Since $\phi$ is unital, we have that $\|\phi\|_{\dec} =1$.

\noindent $(3) \Longrightarrow (1)$.  Suppose that $\| \phi \|_\dec =1$, and let $p \in V^+$. If $\phi(p) \notin W^+$. Then there exists a state $f$ on $W$ such that, $\alpha = f(\phi(p)) \notin \mathbb{R}^+$.  As in the proof of $(2) \Longrightarrow (1)$ choose $c \in \C$ and $r > 0$ such that $|c - \alpha| > r$, and $0 \le t \le \|p\|$ implies $|c-t| \leq r$. Then we have that 
\begin{align*}
| c - \alpha  | &= |f (ce' - \phi(p))| \leq \| ce' - \phi (p) \|_m \leq  \| ce' - \phi(p) \|_\dec \\
&= \| \phi ( ce - p) \|_\dec \leq \| ce - p \|_\dec  \ \text{ (since $\| \phi \|_\dec \leq 1$)} \\
&= \| ce - p \|_m \ \text{ (by Lemma~\ref{technical-min-dec-lem})} \\
&= \sup \{ | c-g(p) | : g \text{ is a state on $V$} \} \\
&\leq r   \quad \text{ (since $0 \leq g(p) \leq \| p \|$)}
\end{align*}
which contradicts the fact that $|c - \alpha| > r$.  Thus we must have that $\phi(p) \in W^+$, and $\phi$ is positive.
\end{proof}

\section{Examples of ordered $*$-vector spaces} \label{Examples-sec}

In this section we introduce some of the key examples of ordered *-vector spaces with an Archimedean order unit and compare the minimal, maximal, and decomposition norms to some other well-known norms.

\subsection{Function systems and a complex version of Kadison's Theorem} \label{Funct-Sys-subsec}

\begin{definition}
Let $X$ be a compact Hausdorff space.  A (concrete) \emph{function system} is a self-adjoint subspace of $C(X)$ that contains the identity function.
\end{definition}

Notice that $C(X)$ is an ordered $*$-vector space with $C(X)^+ = \{ f \in C(X) : f(x) \geq 0 \text{ for all } x \in X \}$, and the constant function $1$ is an Archimedean order unit for this space.  Furthermore, if $V \subseteq C(X)$ is a function system, then $V$ is an ordered $*$-vector space with $V^+ := V \cap C(X)^+$ and $1$ is an Archimedean order unit for $V$.  We shall show that all ordered $*$-vector spaces with an Archimedean order unit arise in this fashion.

The following result is the complex version of Kadison's characterization of real function systems \cite[Theorem~II.1.8]{Alf}.

\begin{theorem} \label{characterize-func-sys}
Let $(V,V^+)$ be an ordered $*$-vector space with Archimedean order unit $e$.  Give $V$ the order topology of Definition~\ref{complex-order-top-def}, and endow the state space $S(V) := \{ f: V \to \C : \text{ $f$ is a state } \}$ with the corresponding weak-$*$ topology.  Then $S(V)$ is a compact space, and the map $\Phi : V \to C(S(V))$ given by $\Phi(v) (f) := f(v)$ is an injective map that is an order isomorphism onto its range with the property that $\Phi(e) = 1$.  Furthermore, $\Phi$ is an isometry with respect to the minimal order norm on $V$ and the sup norm on $C(S(V))$.
\end{theorem} 

\begin{proof}
If we endow $V$ with any order norm $||| \cdot |||$ then it follows from Lemma~\ref{pos-fct-norm-f(e)-lem} that $S(V)$ is a subset of the unit ball in $V^*$.  In addition, suppose that $\{ f_\lambda \}_{\lambda \in \Lambda} \subseteq S(V)$ is a net of states, and $\lim f_\lambda = f$ in the weak-$*$ topology for some $f \in V^*$ .  Then for any $v \in V^+$ we have that $\lim f_\lambda(v) = f(v)$, and since $f_\lambda(v)$ is non-negative for all $n$, it follows that $f(v) \geq 0$ for all $v \in V^+$.  Hence $f$ is a positive functional.  In addition $f(e) = \lim f_\lambda(e) = \lim 1 = 1$ so that $f$ is a state.  Thus $S(V)$ is closed in the weak-$*$ topology.  It follows from Alaoglu's Theorem \cite[Theorem~III.3.1]{Co2} that the unit ball in $V^*$ is compact in the weak-$*$ topology, and since $S(V)$ is a closed subset of the unit ball, we have that $S(V)$ is  compact in the weak-$*$ topology.

Consider the map $\Phi : V \to C(S(V))$ given by $\Phi(v) (f) := f(v)$.  If $\Phi(v) = 0$, then $f(v) = 0$ for all states $f : V \to \C$ and it follows from Proposition~\ref{states-sep-pts} that $v = 0$.  Thus $\Phi$ is injective.  

In addition, if $v \in V^+$, then for any $f \in S(V)$ we have that $\Phi(v) (f) = f(v) \geq 0$ by the positivity of $f$.  Hence the function $\Phi(v)$ takes on non-negative values and $\Phi(v) \in C(S(V))^+$.  Conversely, if $\Phi(v) \in C(S(V))^+$, then for all $f \in S(V)$ we have that $f(v) = \Phi(v)(f) \geq 0$, and thus $v \in V^+$ by Proposition~\ref{complex-states-pos-imply-pos}.  Therefore $\Phi$ is an order isomorphism onto its range.

Finally, if $v \in V$, then 
\begin{align*}
\| v \|_m &= \sup \{ |f(v)| : f : V \to \C \text{ is a state} \} \\
&= \sup \{ | \Phi(v)(f) | : f \in S(V) \} \\
&= \| \Phi(v) \|_\infty
\end{align*}
so that $\Phi$ is an isometry with respect to the minimal order norm on $V$ and the sup norm on $C(S(V))$.
\end{proof}

The above theorem gives an abstract characterization of function spaces.  Hence we make the following definition.

\begin{definition}
A \emph{function system} is an ordered $*$-vector space with an Archimedean order unit.
\end{definition}

Theorem~\ref{characterize-func-sys} shows that any function system is order isomorphic to a self-adjoint unital subspace of $C(X)$ via an isomorphism that is isometric with respect to the minimal norm and takes the Archimedean order unit to the constant function $1$.  This characterization is useful because it allows us to view any ordered $*$-vector space with an Archimedean order unit as a subspace of a commutative $C^*$-algebra.

\subsection{Unital $C^*$-algebras} \label{unital-op-alg-subsec}

Let $A \subseteq B(H)$ be a unital $C^*$-algebra with unit $e = I_{H}$, so that $A$ is also an Archimedean ordered $*$-vector space with $A^+$ equal to the usual cone of positive elements. We use $\| \cdot \|_{\textnormal{op}}$ to denote the usual $C^*$-algebra norm in order to distinguish it from the order norm $\| \cdot \|$ on $A_h$ and from the minimal, maximal, and decomposition extensions of this norm to $A$. Note that in this setting the definition of a state given in this paper agrees with the usual definition of a state on a $C^*$-algebra.

\begin{proposition} \label{min-op-norms-coincide}
If $X \in A$ is a normal operator (i.e., $XX^* = X^*X$), then $\|  X \|_m = \| X \|_{\textnormal{op}}$.
\end{proposition}

\begin{proof}
By the definition of the minimal order norm we have that $\| X \|_m := \sup \{ |f(X)| : \, f : A \to \C \text{ is a state} \}$.  But for any state $f : A \to \C$ we have that $|f(X)| \leq \| f \| \|X \|_{\textnormal{op}} \leq \|X \|_{\textnormal{op}}$.  Thus $\|  X \|_m \leq \| X \|_{\textnormal{op}}$.  Furthermore, since $X$ is normal, it follows from \cite[Theorem~3.3.6]{Mur} that there exists a state $f : A \to \C$ with $|f(X)| = \| X \|_{\textnormal{op}}$.  Thus $\|  X \|_m = \| X \|_{\textnormal{op}}$.
\end{proof}

\begin{corollary} \label{com-minimal-op-norm-equal}
If $A$ is commutative, then $\|  \cdot \|_m = \| \cdot \|_{\textnormal{op}}$.
\end{corollary}

\begin{corollary}
If $X \in A$ is self-adjoint, then $\| X \|_\textnormal{op} = \| X \|$.  Thus the operator norm $\| \cdot \|_\textnormal{op}$ on $A$ restricts to the order norm $\| \cdot \|$ on $A_h$.
\end{corollary}

Since $\| X \|_\textnormal{op} = \| X^* \|_\textnormal{op}$, we see that $\| \cdot \|_\textnormal{op}$ is a $*$-norm.  Because $\| \cdot \|_\textnormal{op}$ restricts to the order norm on $A_h$, this shows that $\| \cdot \|_\textnormal{op}$ is an order norm on $(A, A^+)$.  Thus we have the following.

\begin{corollary}
The operator norm $\| \cdot \|_{\textnormal{op}}$ is an order norm on $A$.  Consequently, for all $X \in A$ we have that $\| X \|_m \leq \| X \|_{\textnormal{op}} \leq \| X \|_M$. 
\end{corollary}

Recall that for $X \in B(H)$ the {\em numerical radius} of $X$, denoted $w(X)$, is defined by
$$w(X) := \sup \{ | \langle Xh,h \rangle |: h \in H \text{ and } \|h\|=1 \}.$$

\begin{proposition} Let $A \subseteq B(H)$ be a $C^*$-algebra containing the unit $e=I_H$. Then for any $X \in A$, we have $\|X\|_m = w(X)$.
\end{proposition}
\begin{proof} 
For any $h \in H$ with $\| h \| = 1$ define $s_h : A \to \C$ by $s_h (X) := \langle Xh, h \rangle$.  One can easily verify that $s_h$ is a state on $A$.  Thus 
\begin{align*}
w(X) &= \sup \{ s_h(X) : h \in H \text{ and } \|h\|=1 \} \\
&\leq \sup \{ |f(X)| : f: A\to \C \text{ is a state } \} = \| X \|_m.\end{align*}
Conversely, it is straightforward to show that $w(X)$ is a $*$-norm on $A$.  In addition, if $X \in A_h$ then it follows from \cite[Theorem~II.2.13]{Co2} that $w(X) = \| X \|_\textnormal{op}$.  Thus $w(X)$ is an order norm on $A$, and it follows from Theorem~\ref{min-norm} that $\| X \|_m \leq w(X)$.  Hence $\| X\|_m=w(X)$.
\end{proof}

\begin{corollary}
If $X \in A$ is normal, then $w(X) = \| X \|_\textnormal{op}$.
\end{corollary}

We now turn our attention to the decomposition norm.

\begin{proposition} Let $A \subseteq B(H)$ be a unital C*-algebra. Then for every $X \in A,$ we have that $\|X \|_{\textnormal{op}} \leq \|X\|_{\dec}$.
\end{proposition}
\begin{proof}
Let $X = \sum_i \lambda_i P_i$ with $P_i \ge 0$ and $\lambda_i \in \mathbb{C}$.
For any unit vectors, $h,k \in H,$ we have that
\begin{align*}
&|\langle Xh,k \rangle| \\
\leq &\sum_i |\lambda_i||\langle P_i^{1/2}h, P_i^{1/2}k \rangle| \leq \sum_i |\lambda_i| \| P_i^{1/2} h \| \, \| P_i^{1/2} k \| \\
\leq &\sqrt{ \sum_i |\lambda_i| \|P_i^{1/2}h\|^2} \sqrt{ \sum_i |\lambda_i| \|P_i^{1/2}k\|^2} = \sqrt{ \sum_i \langle |\lambda_i|P_ih,h \rangle}
\sqrt{\sum_i \langle |\lambda_i|P_i k,k \rangle } \\
= &\left\langle \left( \sum_i | \lambda_i | P_i \right) h, h \right\rangle^{1/2}  \left\langle \left( \sum_i | \lambda_i | P_i \right) k, k \right\rangle^{1/2} \le \left\| \sum_i
|\lambda_i| P_i \right\|_{\textnormal{op}}. 
\end{align*}
Taking the infimum over the right hand side yields $|\langle Xh,k \rangle| \le \|X\|_{\dec}$, and hence $\|X\|_\textnormal{op} \leq \|X\|_{\dec}.$
\end{proof}

\begin{proposition} Let $A$ be a unital commutative C*-algebra, then
$\|f\|_m = \|f\|_{\dec}$ for every $f \in A.$
\end{proposition}
\begin{proof} Write $A=C(X)$ for some compact Hausdorff space $X$, and let $f \in C(X).$  Given any $\epsilon >0,$ we may pick a finite open cover $\{U_i \}_{i=1}^n$ of $X$ and points $x_i \in U_i$ such that $|f(x) - f(x_i)| < \epsilon$ for $x \in U_i.$ We may now take a partition of unity subordinate to this open cover consisting of non-negative continuous functions $\{p_i \}_{i=1}^n$ such that $\sum_{i=1}^n p_i(x) =1$ for every $x \in X$, and with $p_i(x) =0$ when $x \notin U_i.$  If we set $\lambda_i := f(x_i)$, then we have that
$$\left| f(x) - \sum_{i=1}^n \lambda_i p_i(x) \right| = \left| \sum_{i=1}^n (f(x) - \lambda_i)p_i(x) \right| \leq \sum_{i=1}^n |f(x) - \lambda_i|p_i(x).$$
Now this latter sum can be seen to be less than $\epsilon,$ since if $p_i(x) \ne 0,$ then $x \in U_i$ and hence $|f(x) - \lambda_i| < \epsilon.$
If we let $g =f - \sum_i \lambda_i p_i$, then $\|g\| < \epsilon,$ and hence
$$\| f \|_{\dec} \leq \left\| \sum_{i=1}^n \lambda_i p_i \right\|_{\dec} + \|g\|_{\dec} \leq \left\| \sum_{i=1}^n |\lambda_i| p_i \right\| + 2\|g\|_m  \leq \|f\|_m + 2 \epsilon.$$
Since $\epsilon$ was arbitrary, the proof is complete.
\end{proof}

\begin{remark}
To summarize some of the above results:  If $A$ is a unital $C^*$-algebra, then for any $a \in A$ we have that $\| a \|_m \leq \| a \|_\textnormal{op} \leq \| a \|_\dec \leq \| a \|_M$.  Furthermore, if $A$ is commutative we also have that $\| \cdot \|_m =  \| \cdot \|_\textnormal{op} = \| \cdot \|_\dec$.
\end{remark}

\subsection{A characterization of equality of $\| \cdot \|_m$ and $\| \cdot \|_M$} \label{min-max-equal-subsec}

\begin{theorem} \label{char-min-max-equal-thm}
Let $(V,V^+)$ be an ordered $*$-vector space with Archimedean order unit $e$.   Then the minimal order norm $\| \cdot \|_m$ and the maximal order norm $\| \cdot \|_M$ are equal if and only if $V$ is isomorphic to the complex numbers $\C$.
\end{theorem}

\begin{proof}
If $V \cong \C$, then since $\C$ is a unital commutative $C^*$-algebra, it follows from Corollary~\ref{com-minimal-op-norm-equal} that the minimal order norm coincides with the operator norm.  In addition, for any $v \in \C$ we see that if we write $v = v \cdot 1$, then $v \in \C$ and $1 \in \R$ and by the definition of the maximal order norm we have that $\| v \|_M \leq | v | \| 1 \| = |v|$.  By the maximality of the maximal order norm, we then have that $\| v \|_M = | v |$, and thus the maximal order norm and the operator norm coincide.

Conversely, suppose that $V \ncong \C$.  We shall show that there exists $v \in V$ such that $\| v \|_m < \| v \|_M$.  By Theorem~\ref{characterize-func-sys} we may, without loss of generality, assume that $V$ is a subspace of $C(X)$ for some compact Hausdorff space $X$ with the property that $V = V^*$ and $1 \in V$, and furthermore, the minimal order norm $\| \cdot \|_m$ on $V$ coincides with the sup norm $\| \cdot \|$ on $C(X)$.

Since $V \neq \C \cdot 1$, there exists $f \in V$ with $f \geq 0$ and $f \notin \C \cdot 1$.  By normalizing, we may also also assume that $\| f \| = 1$.  Let $h := f-1$, and define $g := (1/\| h \|) h$.  (Note that since $f \neq 1$ we have that $\| h \| > 0$.)  Then $g \in V$ and since $0 \leq f(x) \leq 1$ for all $x \in X$ we see that $h(x) \leq 0$ and $-1 \leq g(x) \leq 0$ for all $x \in X$.  Furthermore, note that for any $x \in X$ we have that $f(x) = 1$ implies that $g(x) = 0$.  Likewise, if $g(x) = -1$, then $h(x) = -\| h \|$ and $f(x) = 1 - \| h \| < 1$.  Because, $0 \leq f \leq 1$ and $ -1 \leq g \leq 0$ , it follows that for any $x \in X$ we have that $|f(x)| = 1$ implies that $g(x) = 0$; and we also have that $|g(x)| = 1$ implies that $f(x) = 1 - \| h \| < 0$.  If we define $F : X \to \C$ by $F(x) := f(x) + i g(x)$, then the previous sentence implies that for all $x \in X$ we have $$| F(x) | = |f(x) + i g(x)| = \sqrt{ |f(x)|^2 + |g(x)|^2} < \sqrt{ 1^2 + 1^2} = \sqrt{2}.$$  Thus $\| F \|_m = \| F \| < \sqrt{2}$.

We shall now show that $\| F \|_M \geq \sqrt{2}$.  Suppose that $F = \sum_{j=1}^n \lambda_j s_j$ for $\lambda_j \in \C$ and real-valued continuous functions $s_j : X \to \R$, and set $$K := \sum_{j=1}^n | \lambda_j | \, \| s_j \|.$$  Since any $\lambda_j$ may be written as $e^{i\theta_j} | \lambda_j |$ for some $\theta_j \in [0, 2\pi)$, we have that $\lambda_j s_j = e^{i \theta_j} (| \lambda_j| s_j)$ and by setting $r_j := |\lambda_j| s_j$ we may assume that $F = \sum_{j=1}^n e^{i\theta_j} r_j$ for real-valued functions $r_j : X \to \R$ and $$K = \sum_{j=1}^n |e^{i\theta_j} | \| r_j \| = \sum_{j=1}^n \| r_j \|.$$  Because $X$ is compact and $f \geq 0$ with $\|f \|=1$, there exists $x_1 \in X$ such that $f(x_1) = 1$.  Likewise, since $-1 \leq g \leq 0$ and $\| g \| =1$ there exists $x_2 \in X$ such that $g(x_2) = -1$.  Since $F = \sum_{j=1}^n e^{i\theta_j} r_j$ and $F(x_1) = f(x_1) + i g(x_1) = 1 + i g(x_1)$ we have that $1 + i g(x_1) =  \sum_{j=1}^n e^{i\theta_j} r_j(x_1)$.  Equating the real parts of this equation gives $1 = \sum_{j=1}^n ( \cos \theta_j )r_j (x_1)$.  Thus $$1 \leq \sum_{j=1}^n | \cos \theta_j | \, | r_j (x_1) |.$$  Likewise, since $F = \sum_{j=1}^n e^{i\theta_j} r_j$ and $F(x_2) = f(x_2) + i g(x_2) = f(x_2) - i$ we have that $f(x_2) - i =  \sum_{j=1}^n e^{i\theta_j} r_j(x_2)$.  Equating the imaginary parts of this equation gives $-1 = \sum_{j=1}^n ( \sin \theta_j )r_j (x_2)$.  Thus $$1 \leq \sum_{j=1}^n | \sin \theta_j | \, | r_j (x_2) |.$$  Hence we have that 
\begin{align*}
2 &= 1 + 1 \\
&\leq  \sum_{j=1}^n | \cos \theta_j | \, | r_j (x_1) | +  \sum_{j=1}^n | \sin \theta_j | \, | r_j (x_2) | \\
&\leq \sum_{j=1}^n ( |\cos \theta_j | + | \sin \theta_j | ) \max \{ |r_j(x_1)|, |r_j(x_2)| \} \\
&\leq \sum_{j=1}^n \sqrt{2} \max \{ |r_j(x_1)|, |r_j(x_2)| \} \\
&\leq \sum_{j=1}^n \sqrt{2}\| r_j \| \\
&= \sqrt{2} K.
\end{align*}
It follows that $K \geq 2/\sqrt{2} = \sqrt{2}$, and since $\| F \|_M$ equals the infimum taken over all such $K$, we have that $\| F \|_M \geq \sqrt{2}$.  Thus we have shown that $F$ is an element of $V$ with $\| F \|_m < \| F \|_M$.
\end{proof}

\subsection{Convex combinations of $\| \cdot \|_m$ and $\| \cdot \|_M$} \label{convex-comb-subsec}

Let $(V,V^+)$ be an ordered $*$-vector space with Archimedean order unit $e$.  If $V \cong \C$, then $\| \cdot \|_m = \| \cdot \|_M$, and there is a unique order norm on $V$.  When $V \ncong \C$, Theorem~\ref{char-min-max-equal-thm} shows that $\| \cdot \|_m$ is not equal to $\| \cdot \|_M$.  From this one can deduce that there are infinitely many order norms on $V$:  For each $0 \leq t \leq 1$ define $\| \cdot \|_t$ by $$ \| v \|_t := t \| v \|_m + (1-t) \| v \|_M.$$ Then $\| \cdot \|_t$ is a convex combination of the minimal and maximal order norms, and it is straightforward to show that $\| \cdot \|_t$ is an order norm and that the $\| \cdot \|_t$'s are distinct for each $t$.  Thus, when $V \ncong \C$, there are at least a continuum of order norms on $V$.  It is natural to ask whether these are the only order norms.  The following examples show that, in general, the decomposition norm $\| \cdot \|_\dec$ is not a convex combination of the minimal and maximal order norms.  

\begin{example}
Consider $\C^2$ with the usual positive elements $\C^2 = \{ (x,y) : x \geq 0 \text{ and } y \geq 0 \}$ and Archimedean order unit $e = (1,1)$.  Let $\| \cdot \|_m^{\C^2}$,  $\| \cdot \|_\dec^{\C^2}$, and $\| \cdot \|_M^{\C^2}$ denote the minimal, decomposition, and maximal order norms on $\C^2$, respectively.  Since $\C^2$ is a commutative $C^*$-algebra, we have that $\| (z_1, z_2 ) \|_m^{\C^2} = \| (z_1, z_2 ) \|_\dec^{\C^2} = \max \{ |z_1|, |z_2| \}$.  Thus $\| (1, i ) \|_m^{\C^2} = \| (1, i ) \|_\dec^{\C^2} = 1$.  In addition, an argument similar to that in the end of the proof of Theorem~\ref{char-min-max-equal-thm} shows that $\| (1,i) \|_M^{\C^2} \geq \sqrt{2}$, and since $(1,i) = (1+i)(1/2, 1/2) + (1-i) (1/2, -1/2)$ and $| (1+i)| \, \| (1/2, 1/2) \| + |(1-i)|  \, \| (1/2, -1/2) \| = \sqrt{2} (1/2) + \sqrt{2} (1/2) = \sqrt{2}$ we may conclude that $\| (1,i) \|_M^{\C^2} = \sqrt{2}$.
\end{example}

\begin{example}
Consider the $C^*$-algebra $M_2(\C)$ with the usual positive elements and Archimedean order unit $e = I_2$.  Let $\| \cdot \|_m^{M_2}$,  $\| \cdot \|_\dec^{M_2}$, and $\| \cdot \|_M^{M_2}$ denote the minimal, decomposition, and maximal order norms on $M_2(\C)$, respectively.  One can show that for any $\lambda \in \C$ we have $\left\| \left( \begin{smallmatrix} 0 & \lambda \\ 0 & 0 \end{smallmatrix} \right) \right\|_m^{M_2} = \frac{| \lambda |}{2}$, and $\left\| \left( \begin{smallmatrix} 0 & \lambda \\ 0 & 0 \end{smallmatrix} \right) \right\|_\dec^{M_2} = \left\| \left( \begin{smallmatrix} 0 & \lambda \\ 0 & 0 \end{smallmatrix} \right) \right\|_M^{M_2} = | \lambda |$.  In addition, one can show that for any $\lambda, \mu \in \C$ we have $\left\| \left( \begin{smallmatrix} \lambda & 0 \\ 0 & \mu \end{smallmatrix} \right) \right\|_\dec^{M_2} = \| ( \lambda, \mu ) \|_\dec^{\C^2}$, and $\left\| \left( \begin{smallmatrix} \lambda & 0 \\ 0 & \mu \end{smallmatrix} \right) \right\|_M^{M_2} = \| ( \lambda, \mu ) \|_M^{\C^2}$.  It follows from these equalities that
\begin{equation} \label{min-dec-max-1}
\left\| \begin{pmatrix} 0 & 1 \\ 0 & 0 \end{pmatrix} \right\|_m^{M_2} = \frac{1}{2} \quad \text{ and } \quad \left\| \begin{pmatrix} 0 & 1 \\ 0 & 0 \end{pmatrix} \right\|_\dec^{M_2} = \left\| \begin{pmatrix} 0 & 1 \\ 0 & 0 \end{pmatrix}  \right\|_M^{M_2} = 1.
\end{equation}
It also follows that
\begin{equation} \label{min-dec-max-2}
\left\|  \begin{pmatrix} 1 & 0 \\ 0 & i \end{pmatrix} \right\|_\dec^{M_2} = \| ( 1, i ) \|_\dec^{\C^2} = 1 \text{ and } \left\| \begin{pmatrix} 1 & 0 \\ 0 & i \end{pmatrix} \right\|_M^{M_2} = \| ( 1, i ) \|_M^{\C^2} = \sqrt{2}.
\end{equation}
One can see from Eq.~\ref{min-dec-max-1} and Eq.~\ref{min-dec-max-2} that the three norms $\| \cdot \|_m^{M_2}$,  $\| \cdot \|_\dec^{M_2}$, and $\| \cdot \|_M^{M_2}$ are distinct, and furthermore, that $\| \cdot \|_\dec^{M_2}$ is not a convex combination of $\| \cdot \|_m^{M_2}$ and $\| \cdot \|_M^{M_2}$.
\end{example}

\end{document}